\documentclass[11pt]{article}      
\usepackage{amsmath,amssymb, amsthm, eucal,bm,accents,mathtools}  
\usepackage{graphicx, amscd}
\usepackage{framed}
\usepackage{url}               
\usepackage{rotating}
\usepackage{subdepth}

\usepackage{tensor}

 \usepackage{relsize} 
 \usepackage{xfrac}  

\usepackage{epstopdf}

\usepackage{txfonts}                  
\usepackage[T1]{fontenc}     

\usepackage{scrextend}     

\usepackage{enumerate}       
\usepackage{float}                 
\usepackage{tikz}
\usetikzlibrary{matrix, arrows, calc}
\usetikzlibrary{arrows,shapes}
\usetikzlibrary{decorations.markings}
\usetikzlibrary{arrows,shapes,positioning}
\usetikzlibrary{arrows,calc,decorations.markings,math,arrows.meta}

\definecolor{gold}{rgb}{1,.70,.0}   

\theoremstyle{plain}
\newtheorem{theorem}{Theorem}[section]            
\newtheorem{proposition}[theorem]{Proposition}  

\newtheorem{corollary}[theorem]{Corollary}	      %

\theoremstyle{definition}
\newtheorem{definition}[theorem]{Definition}    
\newtheorem{remark}[theorem]{Remark}

\numberwithin{theorem}{section}
\numberwithin{equation}{section}
\newcommand{\gaction}[2]{\genfrac{}{}{0.5pt}{}{#1}{#2}%
                        \!\lower2pt\hbox{\rotatebox[origin=c]{-90}{{$\looparrowright$}}}}
\newcommand{\dotfraction}[2]{\genfrac{}{}{0.5pt}{}{#1}{#2}%
                        \!\lower.5pt\hbox{{$\circ$}}}

\usepackage{titlesec}                                                           
\titlelabel{\thetitle.\ }                                                         
\titleformat*{\section}{\fontsize{14pt}{14pt} \bf}        
\usepackage{fancyhdr}
\usepackage{sidecap}  
\usepackage[hang,small,sc]{caption}

\def\z{\hbox{\raisebox{-.1pt}{-}}}

\usepackage{fancybox}

\def\id{\hbox{\rm id}\,}

\def\gen{\hbox{\rm gen}\,}

\def\tspin{\hbox{\rm spin}}
\def\spin{\hbox{\rm spin\,}}

\def\del{\partial}
\def\Inv{\hbox{\rm Inv}}

\def\Liff{\Leftrightarrow}

\def\Lthenif{\Leftarrow}

\def\spin{\hbox{spin\,}}

\usepackage{frcursive}

\def\span{{\rm span}}

\def\APq{\hbox{\sc Apo}\mkern1mu( 4)}
\def\SIq{\hbox{\sc Kal}\mkern1mu(4)}
\def\APt{\hbox{\sc Apo}^*\!(3)}
\def\SIt{\hbox{\sc Kal}\mkern1mu(3)}
\def\APqq{\hbox{\sc Apo}\mkern1mu (5)}
\def\SIqq{\hbox{\sc Kal}\mkern1mu(5)}
\def\SI{\hbox{\sc Kal}}
\def\AP{\hbox{\sc Apo}}
\def\Des{\hbox{\sc Des}}
\def\Dest{\hbox{\sc Des}^*\!(3)}
\def\Desq{\hbox{\sc Des}\mkern1mu(4)}

\begin{document}

\title{Integral spinors, Apollonian disk packings, \\  and Descartes groups}






\author{Jerzy Kocik 
    \\ \small Department of Mathematics, Southern Illinois University, Carbondale, IL62901
   \\ \small jkocik{@}siu.edu
}


\date{}

\maketitle

\vspace{-.37in}

\begin{abstract}
\noindent
We show that every irreducible integral Apollonian packing can be set in the Euclidean space 
so that all of its tangency spinors and all reduced coordinates and co-curvatures are integral.
As a byproduct, we prove that in any integral Descartes configuration, the sum of the curvatures of 
two adjacent disks can be written as a sum of two squares. 
Descartes groups are defined, and an interesting occurrence of the Fibonacci sequence is found.
\\[5pt]
\\[5pt]
\end{abstract}

\section{Introduction}

Apollonian disk packings, the examples of which are presented 
in Figure \ref{fig:A1} and \ref{fig:A2}, convey a rich number-theoretic content.
To start with, there exist ``integral packings'' where the curvatures of the disks 
(the numbers in the figures) are all integers.
There are infinitely many such packings.
Curvatures are not the only magnitudes associated with the disks.
There are also disk symbols -- related to the disks positions-- and the tangency spinors.
The main question is if the integrality of these quantities (especially the spinors)  is possible.

Section 3 clarifies these terms and introduces a few new convenient concepts.

Section 4 presents the main result: the proof that indeed every Apollonian
\mbox{packing} 
may be located in the plane so that such a ``hyper-integrality'' is attained.

Section 5 proves some properties of the concepts used in the proof.

Section 6  discusses an interesting order of the Apollonian disk packings
that emerged in due process, including an intriguing exceptional case of a ``Fibonacci'' family.

Finally, Section 7 discusses the group-theoretic aspect of these concepts.
In particular, we define the Descartes group and kaleidoscope group, 
and their action on triples and quadruples of tangent disks.

We start with an associated minor, yet intriguing, observation that seem to have been missed, 
namely, that the sum of the curvatures of two adjacent disks in such packings 
can always be presented as a sum of two squares,
as explained in Section~2.
Proving this fact directly may be a rather challenging task,
while we will procure it as a corollary to the main theorem.
 
This paper is companion to \cite{jk-lattices}, where the integral Apollonian packings are 
derived from spinors only.

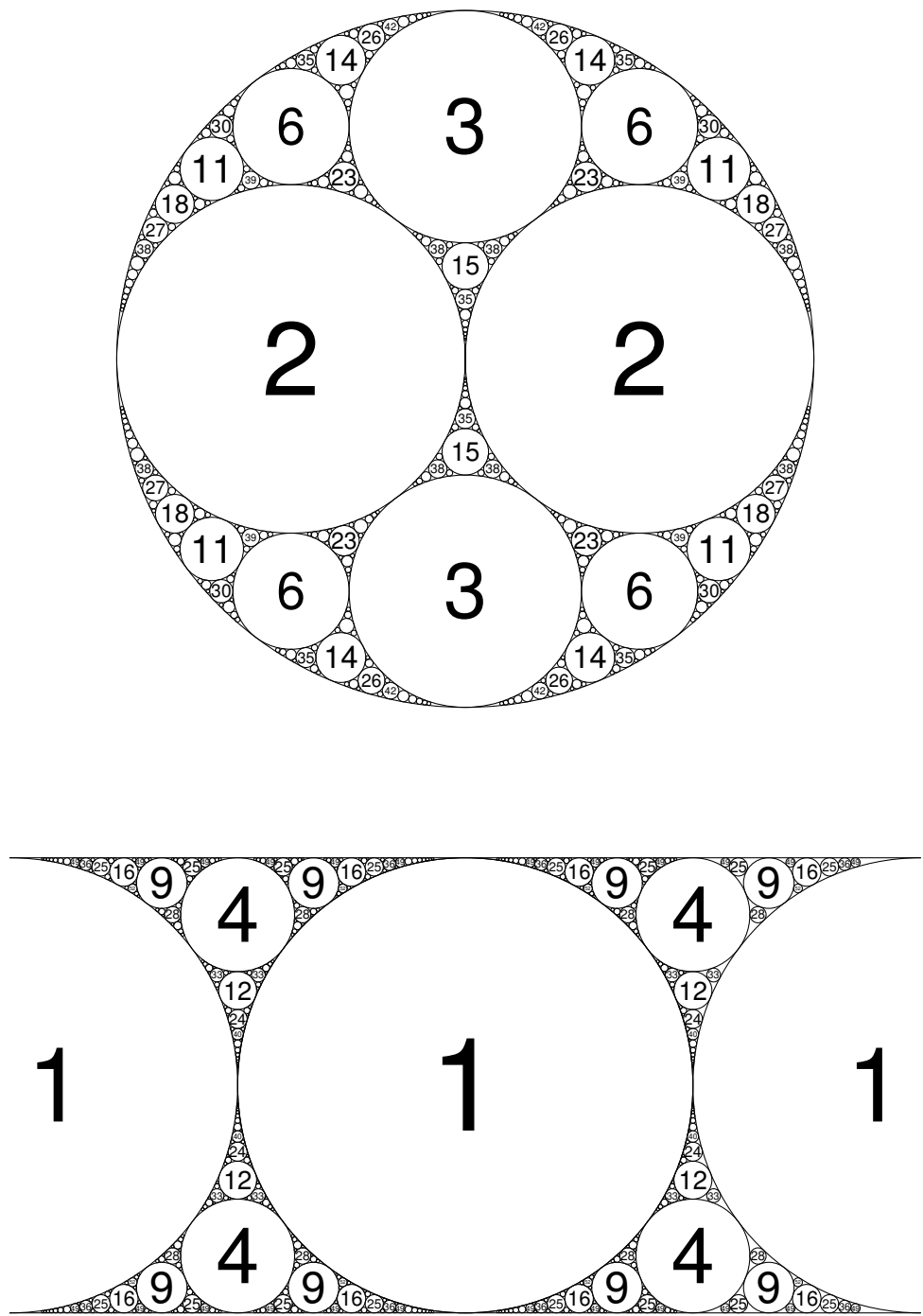
\begin{figure}
\centering

\begin{tikzpicture}[scale=4.9]
\draw [thin] (0,0) circle (1);

\foreach \a/\b/\c/\s   in {
1 / 0 / 2 / 4
}
\draw [thin] (\a/\c,\b/\c) circle (1/\c)
          (-\a/\c,\b/\c) circle (1/\c)
node  [scale=\s]  at (\a/\c,\b/\c)   {$\sf \c$}
node  [scale=\s]  at (-\a/\c,\b/\c) {$\sf \c$};

\foreach \a/\b/\c/\s in {
0 / 2 / 3   /3,
0 /4 /15   /1,
0 / 6 / 35 / 0.5 
}
\draw [thin] (\a/\c,\b/\c) circle (1/\c)
          (\a/\c,-\b/\c) circle (1/\c)
        node [scale= \s]  at (\a/\c,\b/\c)  {$\sf \c$}
        node [scale= \s]  at (\a/\c,-\b/\c) {$\sf \c$};

\foreach \a/\b/\c in {
0 / 8/ 63,
0 /10 / 99,
0 / 12 / 143,
0 / 14 / 195,
0 / 16 / 255,
0 / 18 / 323
}
\draw [thin] (\a/\c,\b/\c) circle (1/\c)  
           (\a/\c,-\b/\c) circle (1/\c);
\foreach \a/\b/\c/\s in {
3 / 4 /6  /2, 
8 / 6 / 11  /1.5,
5 / 12/ 14  /1.2 , 
15/ 8 / 18  /1 ,
8 / 12 / 23 /.9 ,
7 / 24 / 26 /.7,
24/	10/	27 /.7,
21/	20/	30 /.7 ,
16/	30/	35 / .6,
3/	12/	38 /.5,
35/	12/	38 /.5,
24/	20/	39 /.4,
9/	40/	42 /.4
}
\draw [thin] (\a/\c,\b/\c) circle (1/\c)       (-\a/\c,\b/\c) circle (1/\c)
          (\a/\c,-\b/\c) circle (1/\c)       (-\a/\c,-\b/\c) circle (1/\c)
node [scale=\s]  at (\a/\c,\b/\c)    {$\sf \c$}
node [scale=\s]  at (-\a/\c,\b/\c)   {$\sf \c$}
node [scale=\s]  at (\a/\c,-\b/\c)   {$\sf \c$}
node [scale=\s]  at (-\a/\c,-\b/\c)  {$\sf \c$};


\foreach \a /  \b / \c   in {
16/	36/	47, 15/	24/	50,
48/	14/	51, 45/	28/	54,
24/	30/	59, 40/	42/	59,
11/	60/	62, 21/	36/	62,
48/	28/	63, 33/	56/	66,
63/	16/	66, 8/	24/	71,
55/	48/	74, 24/	70/	75,
48/	42/	83, 80/	18/	83,
13/	84/	86, 77/	36/	86,
24/	76/	87, 24/	40/	87,
39/	80/	90, 64/	60/	95,
80/	36/	95, 33/	72/	98,
63/	48/	98, 65/	72/	98,
3/	20/	102, 99/	20/	102,
8/	30/	107, 56/	90/	107,
5/	36/	110, 69/	60/	110,
91/	60/	110, 48/	92/	111,
48/	56/	111, 15/	112/114,
15/	40/	114, 40/	72/	119,
39/	96/	122, 120/22/	123,
117/44/	126, 32/126/	131,
48/102/	131, 112/66/	131,
35/60/	134, 99/60/	134,
120/44/	135, 105/88/	138,
32/132/	143, 80/72/	143,
17/144/	146, 143/24/	146,
48/70/	147, 96/110/	147,
45/76/	150, 51/140/	150,
120/66/	155, 85/132/	158,
63/80/	162, 24/60/	167,
56/120/	167, 136/84/	167,
55/96/	170, 119/120/170,
72/154/171, 168/26/171,
69/92/174, 165/52/174,
64/102/179, 80/90/179,
160/78/179, 19/180/182,
51/156/182, 141/84/182,
120/88/183, 168/52/183,
57/176/186, 153/104/186,
48/84/191, 128/132/191,
65/120/194, 95/168/194,
129/120/194, 144/130/195,
3/28/198, 195/28/	198,
8/42/203
}
\draw [thin] (\a/\c,\b/\c) circle (1/\c)          (-\a/\c,\b/\c) circle (1/\c)
          (\a/\c,-\b/\c) circle (1/\c)         (-\a/\c,-\b/\c) circle (1/\c);
\end{tikzpicture}

~~

~~

~~

\begin{tikzpicture}[scale=3.2, rotate=90]  
\clip (-1.2,-2) rectangle (1.2,2);

\draw (1,-2) -- (1,3);
\draw (-1,-2) -- (-1,3);
\draw (0,0) circle (1);
\draw [thin] (0,0) circle (1)
          node  [scale=5]  at (0,0) {$\sf 1$};
\draw [thin] (0,2) circle (1)
          node  [scale=4]  at (0,1.8) {$\sf 1$};
\draw [thin] (0,-2) circle (1) 
          node  [scale=4]  at (0,-1.8) {$\sf 1$};         
\foreach \a/\b/\c/\s in {
3/  4/ 4/    3,
5/ 12/	12 /  1, 
7/ 24/	24/ .6,
9/ 40/	40/ .3
}
\draw [thin] (\a/\c,\b/\c) circle (1/\c)    (\a/\c,-\b/\c) circle (1/\c)
          (-\a/\c,\b/\c) circle (1/\c)    (-\a/\c,-\b/\c) circle (1/\c)
        node [scale= \s]  at (\a/\c,\b/\c)  {$\sf \c$}
        node [scale= \s]  at (\a/\c,-\b/\c) {$\sf \c$}
        node [scale= \s]  at (-\a/\c,\b/\c) {$\sf \c$}
        node [scale= \s]  at (-\a/\c,-\b/\c) {$\sf \c$};
\foreach \a/\b/\c in {
11/	60/	60,
13/	84/	84,
15/	112/112,
17/	144/144,
19/	180/180,
21/	220/220,
23/	264/264,
25/	312/312
}
\draw [thin] (\a/\c,\b/\c) circle (1/\c)   (\a/\c, -\b/\c) circle (1/\c)
          (-\a/\c, \b/\c) circle (1/\c)   (-\a/\c, -\b/\c) circle (1/\c);

\foreach \a/\b/\c/\d in {
8/    6/   9     / 1.7,
15/   8/   16  / .8,
24/  20/  25 /.6,
24/  10/  25/.5,
21/  20/  28 / .5,
16/	30/	33/ .4,
35/	12/	36/.4,
48/	42/	49/.3,
48/	28/	49/.3,
48/	14/	49/.3,
45/	28/	52/.2
}
\draw [thin] (\a/\c,\b/\c) circle (1/\c)       (-\a/\c,\b/\c) circle (1/\c)
          (\a/\c,-\b/\c) circle (1/\c)       (-\a/\c,-\b/\c) circle (1/\c)
          (\a/\c,2-\b/\c) circle (1/\c)       (-\a/\c,2-\b/\c) circle (1/\c)
          (\a/\c,-2+\b/\c) circle (1/\c)       (-\a/\c,-2+\b/\c) circle (1/\c)
node [scale=\d]  at (\a/\c,\b/\c)    {$\sf \c$}
node [scale=\d]  at (-\a/\c,\b/\c)   {$\sf\c$}
node [scale=\d]  at (\a/\c,-\b/\c)   {$\sf\c$}
node [scale=\d]  at (-\a/\c,-\b/\c)  {$\sf\c$}
node [scale=\d]  at (\a/\c,2-\b/\c)    {$\sf\c$}
node [scale=\d]  at (-\a/\c,2-\b/\c)   {$\sf\c$}
node [scale=\d]  at (\a/\c,-2+\b/\c)    {$\sf\c$}
node [scale=\d]  at (-\a/\c,-2+\b/\c)   {$\sf\c$};

\foreach \a /  \b / \c   in {
40/	42/	57,	33/	56/	64,	63/	48/	64,	63/	16/	64,	55/	48/	72,	24/	70/	73,
69/	60/	76,	80/	72/	81,	64/	60/	81,	80/	36/	81,	80/	18/	81,	77/	36/	84,
39/	80/	88,	65/	72/	96,	48/	92/	97,	99/	60/	100,	99/	20/	100,	56/	90/	105,
91/	60/	108,	120/	110/	121,	120/	88/	121,	120/	66/	121,	120/	44/	121,	120/	22/	121,
117/	44/	124,	32/	126/	129,	112/	66/	129,	105/	88/	136,	143/	120/	144,	143/	24/	144,
96/	110/	145,	51/	140/	148,	141/	84/	148,	136/	120/	153,	136/	84/	153,	149/	132/	156,
85/	132/	156,	129/	120/	160,	119/	120/	168,	168/	156/	169,	72/	154/	169,	168/	130/	169,
168/	104/	169,	168/	78/	169,	168/	52/	169,	168/	26/	169,	165/	52/	172,	128/	132/	177,
160/	78/	177,	57/	176/	184,	153/	104/	184,	95/	168/	192,	96/	186/	193,	144/	130/	193,
189/	148/	196,	195/	140/	196,	195/	84/	196,	195/	28/	196,	40/	198/	201,	104/	180/	201,
184/	162/	201
}
\draw [thin] (\a/\c,\b/\c) circle (1/\c)          (-\a/\c,\b/\c) circle (1/\c)
          (\a/\c,-\b/\c) circle (1/\c)         (-\a/\c,-\b/\c) circle (1/\c)
          (\a/\c,2-\b/\c) circle (1/\c)          (-\a/\c,2-\b/\c) circle (1/\c);

\end{tikzpicture}

\caption{At the top:  Apollonian Window. At the bottom: Apollonian Belt}
\label{fig:A1}
\end{figure}

\begin{figure}
\centering
\begin{tikzpicture}[scale=31]
\draw [thin] (3/6,4/6) circle (1/6);
\draw [thin] (6/11,8/11) circle (1/11);
\draw [thin] (7/14,8/14) circle (1/14);
\draw [thin] (6/15,10/15) circle (1/15);
\foreach \a/\b/\c/\d/\s in {
 6/ 8/ 11/ 94/7,
 7/ 8/ 14/ 8/6,
 6/ 10/ 15/ 9/6,
14/ 14/ 23/ 17/4,
11/ 20/ 26/ 20/4,
14/ 20/ 35/ 17/3,
27/ 28/ 42/ 36/3,
22/ 38/ 47/ 41/2,
30/ 28/ 51/ 33/2,
22/ 44/ 59/ 41/2
}
\draw [thin] (\a/\c,\b/\c) circle (1/\c)
        node [scale= \s]  at (\a/\c,\b/\c)  {$\hbox{\sf \c}\ $}
;

\foreach \a/\b/\c/\d in {
30/ 38/ 71/ 33,
46/ 50/ 71/ 65,
27/ 44/ 74/ 36,
39/ 64/ 78/ 72,
41/ 56/ 86/ 56,
54/ 50/ 95/ 57,
57/ 64/ 102/ 72,
70/ 68/ 107/ 89,
49/ 80/ 110/ 80,
39/ 80/ 110/ 72,
71/ 80/ 110/ 104,
62/ 98/ 119/ 113,
75/ 68/ 122/ 84,
54/ 64/ 123/ 57,
54/ 100/ 123/ 105,
46/ 80/ 131/ 65,
57/ 80/ 134/ 72,
54/ 110/ 143/ 105,
86/ 80/ 155/ 89,
97/ 104/ 158/ 128,
102/ 118/ 159/ 153,
91/ 140/ 170/ 164,
81/ 136/ 174/ 144,
62/ 128/ 179/ 113,
105/ 104/ 182/ 120,
75/ 100/ 186/ 84,
86/ 98/ 191/ 89,
70/ 110/ 191/ 89
}
\draw [thin] (\a/\c,\b/\c) circle (1/\c)  
node [scale= 90/\c]  at (\a/\c,\b/\c)  {$\c$}
;

\foreach \a/\b/\c/\d  in {
134/ 140/ 203/ 185,
71/ 128/ 206/ 104,
81/ 152/ 206/ 144,
135/ 128/ 206/ 168,
102/ 134/ 207/ 137,
99/ 140/ 210/ 140,
139/ 164/ 218/ 212,
102/ 140/ 219/ 137,
110/ 188/ 227/ 209,
142/ 128/ 227/ 161,
126/ 190/ 231/ 225,
126/ 146/ 231/ 161,
126/ 118/ 231/ 129,
102/ 194/ 239/ 201,
105/ 136/ 246/ 120,
111/ 176/ 246/ 176,
97/ 152/ 254/ 128,
161/ 176/ 254/ 224,
147/ 164/ 258/ 188,
174/ 170/ 263/ 225,
91/ 188/ 266/ 164,
155/ 140/ 266/ 164,
102/ 208/ 267/ 201,
150/ 164/ 267/ 185,
126/ 140/ 275/ 129,
137/ 224/ 278/ 248,
126/ 206/ 279/ 209,
182/ 218/ 287/ 281,
126/ 176/ 291/ 161,
126/ 212/ 291/ 209,
102/ 188/ 299/ 153
}
\draw (\a/\c,\b/\c) circle (1/\c);  
\node [scale= 3]  at (.34, .8)  {\sf -6};

\end{tikzpicture}

\begin{tikzpicture}[scale=10,rotate=30]
\draw [thin] (1/2,0) circle (1/2);
\foreach \a/\b/\c/\d/\s in {
2/ 0/ 3/ 1/8,
1/ 0/ 6/ 0/5,
2/ 2/ 7/ 1/5,
2/ -2/ 7/ 1/5,
5/ 4/ 10/ 4/3,
5/ -4/ 10/ 4/3,
10/ 6/ 15/ 2,
10/ -6/ 15/2,
2/ 4/ 19/ 2,
2/ -4/ 19/ 2,
17/ 8/ 22/ 1,
17/ -8/ 22/ 1,
10/ 12/ 27/ 1,
10/ -12/ 27/ 1,
26/ 10/ 31/ 1,
26/ -10/ 31/ 1,
 11/ 4/ 34/ 1,
11/ -4/ 34/ 1,
 2/ 6/ 39/ 1,
 2/ -6/ 39/ 1,
 5/ 12/ 42/ 1,
19/ 12/ 42/ 1,
37/ 12/ 42/ 1,
5/ -12/ 42/ 1,
19/ -12/ 42/ 1,
37/ -12/ 42/ 1
}
\draw [thin] (\a/\c,\b/\c) circle (1/\c)
        node [scale= \s]  at (\a/\c,\b/\c)  {$\hbox{\sf \c} \ $}
;
\foreach \a/\b/\c/\d in {
26/ 20/ 43/ 25,
26/ -20/ 43/ 25,
17/ 24/ 54/ 16,
17/ -24/ 54/ 16,
50/ 14/ 55/ 49,
50/ -14/ 55/ 49,
35/ 20/ 58/ 28,
35/ -20/ 58/ 28,
26/ 30/ 63/ 25,
26/ -30/ 63/ 25,
11/ 12/ 66/ 4,
11/ -12/ 66/ 4,
 2/ 8/ 67/ 1,
50/ 28/ 67/ 49,
 2/ -8/ 67/ 1,
50/ -28/ 67/ 49
}
\draw [thin] (\a/\c,\b/\c) circle (1/\c)  
node [scale= 46/\c]  at (\a/\c,\b/\c)  {$\c$}
;
\foreach \a/\b/\c/\d  in {
65/ 16/ 70/ 64,
65/ -16/ 70/ 64,
10/ 24/ 75/ 9,
10/ -24/ 75/ 9,
 26/ 6/ 79/ 9,
26/ -6/ 79/ 9,
29/ 12/ 82/ 12,
59/ 28/ 82/ 52,
29/ -12/ 82/ 12,
59/ -28/ 82/ 52,
50/ 42/ 87/ 49,
82/ 18/ 87/ 81,
50/ -42/ 87/ 49,
82/ -18/ 87/ 81,
35/ 36/ 90/ 28,
35/ -36/ 90/ 28,
26/ 12/ 91/ 9,
26/ 40/ 91/ 25,
26/ -12/ 91/ 9,
26/ -40/ 91/ 25,
41/ 24/ 94/ 24,
41/ -24/ 94/ 24,
82/ 36/ 99/ 81,
82/ -36/ 99/ 81
}
\draw [thin] (\a/\c,\b/\c) circle (1/\c)  
           (\a/\c,-\b/\c) circle (1/\c);
\node [scale= 3]  at (.22,.58)  {\sf -2}; 
\end{tikzpicture}
\caption{Two other examples of irreducible integral packings. 
The curvature of the exterior disk at the top is $-6$, at the bottom: $-2$.}
\label{fig:A2}
\end{figure}
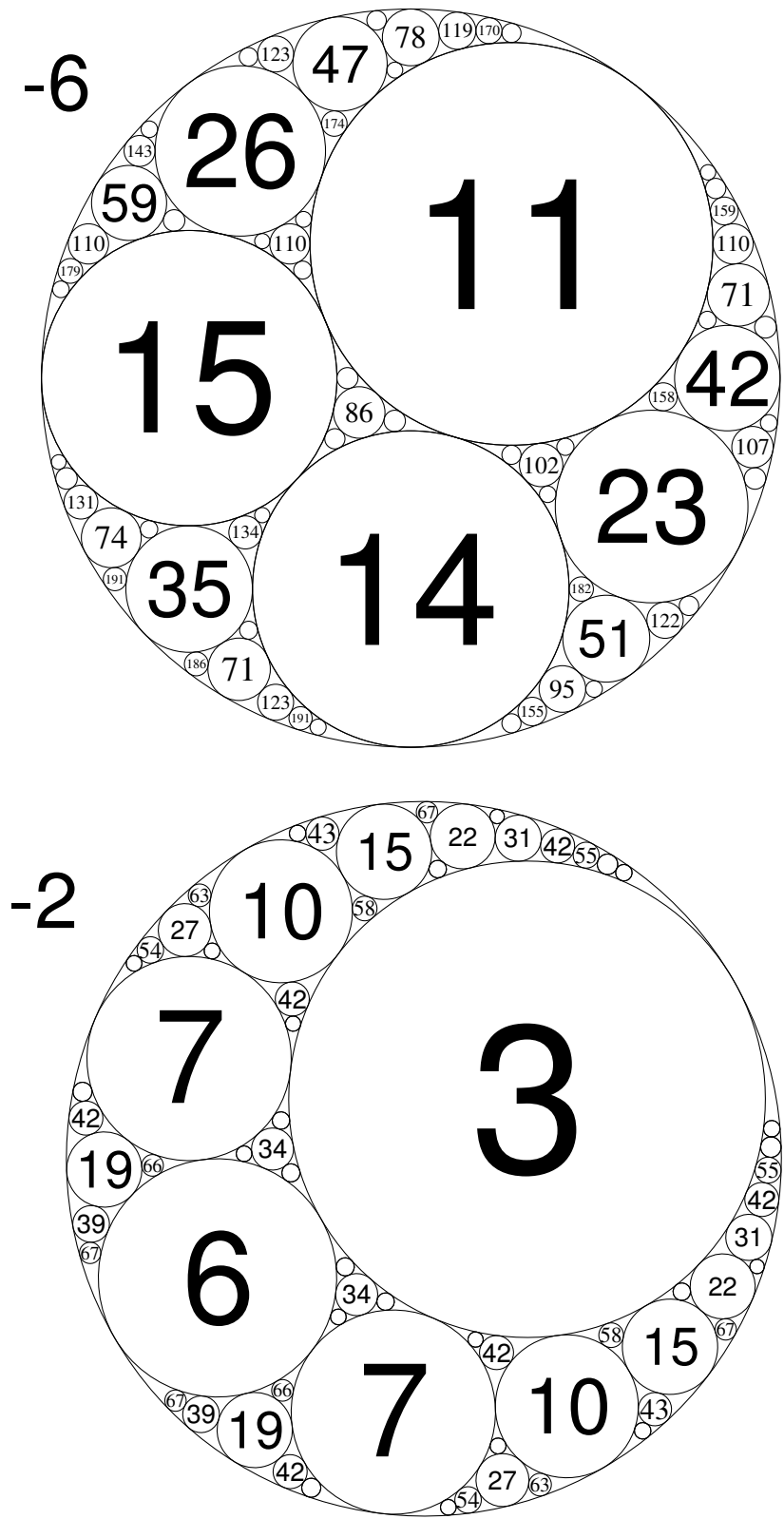

\section{Curious sums} 

Irreducible integral Apollonian disk packings
(see Figures \ref{fig:A1} and~\ref{fig:A2} for examples),
have many interesting number-theoretic properties,  
many of which have been analyzed in various contexts \cite{Lagarias,Lagarias+,Stange}.

But the following feature seems to be overlooked:
the sum of two curvatures of any two disks in contact is always 
a number that can be represented as a sum of squares.
Here are examples read off from the Apollonian Window (Fig.~\ref{fig:A1}):
$$
-1+6 = 5 = 1^2+2^2
\qquad
3+6= 9 = 0^2+3^2
\qquad
6+23 =29=5^2+2^2 
$$
We may reduce this statement to the Descartes configurations,
i.e., configurations of four mutually tangent disks,
and pose it as a purely number-theoretic statement, 
without reference to geometry:

\begin{proposition} 
\label{thm:sum1}
\sf
If a quadruple of non-negative integers $a,b,c,d$ is primitive, i.e, such that  $\gcd(a,b,c,d)=1$, 
and satisfies the quadratic equation:
\begin{equation}
\label{eq:sum1}
(a+b+c+d)^2 = 2(a^2+b^2+c^2+d^2)
\end{equation}
then the sum of any two is a sum of two squares:
\begin{equation}
\label{eq:sum2}
a+b = p^2+q^2
\end{equation}
for some $p, q\in\mathbf Z$.
\end{proposition}

In principle, the property \eqref{eq:sum2} should be provable 
in purely number-theoretic manner from condition \eqref{eq:sum1}.
We leave it as an interesting challenge. 
A verification of this property will  arise as a byproduct of the main result 
in these notes, see Corollary \ref{thm:sum-done}. 
The condition of non-negativity may be relaxed to at most one of $a,b,c,d$ being negative.

~\\
{\bf Remark on the sums of squares:}  
%
Sums of squares were studied by Euler. 
Here are the cases up to 100:  
 1,   2,   4,   5,   8,  9,  10, 13, 16, 17, 18,
20, 25, 26, 29, 32, 34, 36, 37, 40, 41, 45,
49, 50, 52, 53, 58, 61, 64, 65, 68, 72, 73,
74, 80, 81, 82, 85, 89, 90, 97, 98, 100.
 ~ [43 out of 100].
None of the remaining 57 numbers may be found in any integral packings.
The frequency of numbers that can be found 
is decreasing with the magnitude, as 
described by Landau's theorem \cite{Landau,Shanks}.
The density of such numbers among integers drops with $n$
$$
\frac{|S(n)|}{n}  \ \approx\  \frac{b}{\sqrt{\ln n}}  \quad\hbox{for large } n\,,
$$
where $b= 0.764223...$ is known as the Landau-Ramanujan constant 
and $S(n)$ denotes the set of all integers not exceeding $n$ that 
are sums of two squares.
Only one-fifth of the first million numbers can be represented as the sum of squares,
only one-sixth in the first billion, 
and only one in seven in the first trillion.

\section{Terminology and basic facts}

\begin{figure}[h]
\centering
\begin{tikzpicture}[scale=1]
\draw [color=white, fill=gold!30] (-1.2,-1.2) rectangle (1.2,1.2);
\draw [thick,  fill=white] (0,0) circle (1);
\draw [thick, fill=gold!30] (-1/2,0) circle (1/2);
\draw [thick, fill=gold!30] (1/2,0) circle (1/2);
\node at (-1/2, 0) [scale=.9, color=black] {\sf 2};
\node at (1/2, 0) [scale=.9, color=black] {\sf 2};
\node at (.9, .9) [scale=.9, color=black] {\sf -1};
\end{tikzpicture}
\quad
\begin{tikzpicture}[scale=1.7, rotate=30]
\draw [thick,  fill=gold!30] (1/2,0) circle (1/2);
\draw [thick, fill=gold!30] (0,2/3) circle (1/3);
\draw [thick, fill=gold!30] (3/6,4/6) circle (1/6);
\node at (1/2, 0) [scale=.9, color=black] {\sf 2};
\node at (0, 2/3) [scale=.9, color=black] {\sf 3};
\node at (3/6, 4/6) [scale=.9, color=black] {\sf 6};
\end{tikzpicture}
\quad
\begin{tikzpicture}[scale=.7]
\draw [color=white, fill=gold!30] (-2,-1.7) rectangle (2,1.7);
\draw [color=white, fill=white] (-2,-1) rectangle (2,1);
\draw [thick] (-1.7,-1) -- (1.7,-1);
\draw [thick] (-1.7,1) -- (1.7,1);
\draw [thick,  fill=gold!30] (0,0) circle (1);
\node at (0, 0) [scale=.9, color=black] {\sf 1};
\node at (1, 1.27) [scale=.9, color=black] {\sf 0};
\node at (1, -1.27) [scale=.9, color=black] {\sf 0};
\end{tikzpicture}
\quad
\begin{tikzpicture}[scale=.45]
\draw [color=white, fill=gold!30] (-3,-3) rectangle (3,3);
\draw [thick,  fill=white] (0,0) circle (2.1547);
\draw [thick, fill=gold!30] (1, -.57735) circle (1);
\draw [thick, fill=gold!30] (-1, -.57735) circle (1);
\draw [thick, fill=gold!30] (0,  1.1547) circle (1);
\node at (1,  -.57735) [scale=.9, color=black] {\sf 1};
\node at (-1, -.57735) [scale=.9, color=black] {\sf 1};
\node at (0,  1.1547) [scale=.9, color=black] {\sf 1};
\node at (2, 2.45) [scale=.8, color=black] {$\sf 3 - 2\sqrt{\sf 3}$};
\end{tikzpicture}
\caption{Examples of tricycles. The first three are proper integral. the last, (1,1,1),  is not.}
\label{fig:tricycles}
\end{figure}
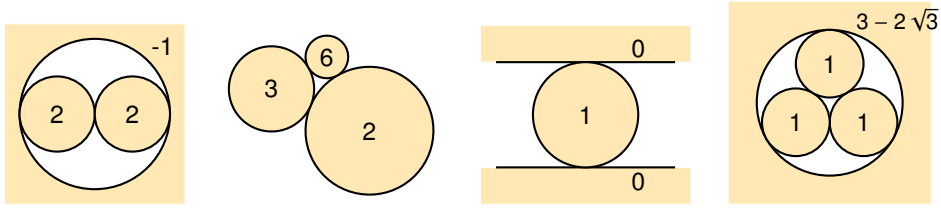


\begin{itemize}
 \setlength\itemsep{-.22em}
\item
A disk and its curvature 
 will usually be denoted by the same (capital) letter, like $A$, $B$, $C$, etc. 
 The disk that extends outside the  circle bounding it has a negative curvature.
A {\bf disk symbol} of a disk $A$ is fraction-like entity 
\begin{equation}
\label{eq:symbol}
\mathbf A = \frac{\dot x, \dot y}{A}
\end{equation}
The meaning (decoding) of the symbol of disk $A$ is 
$$
\hbox{the center} = (x,y) =  \left(\frac{\dot x}{A},\, \frac{\dot y}{A}\right)\,,
\qquad 
\hbox{radius} = \frac{1}{A}
$$
The symbol may be extended to 
 the vector $\mathbf A$ in Minkowski space \cite{jk-Descartes,jk-simple}
that will be presented in two ways:
\begin{equation}
\label{eq:symbol1}
\frac{\dot x,\, \dot y}{A}
\ \mapsto \ 
\frac{\dot x,\, \dot y}{A, A^c}\ 
\ = \  
\begin{bmatrix} \dot x\\ \dot y\\A\\ A^c
\end{bmatrix} \,,
\end{equation} 
where
 $\dot x=x/r=Ax$, \   $\dot y = y/r = Ay$ will be called {\bf reduced coordinates}, 
the curvature $A=1/r$ is the reciprocal radius (possibly negative if the disk extends outside a circle),
and $A^c$, which coincides with the curvature of the disk under inversion in the unit circle,
will be called {\bf co-curvature} of the disk 
 (we follow the terminology of \cite{jk-Descartes}). 
The inner product in the Minkowski space corresponds to the quadratic form 
$Q=-\dot x^2 -\dot y^2 +A A^c$ and can be expressed by Gram matrix $G$:
\begin{equation}
\label{eq:G}
G = \begin{bmatrix} 
-1 & 0 & 0 & 0\\
0 &-1 & 0 & 0\\
0 & 0 & 0 & 1/2\\
0 & 0 & 1/2 & 0
\end{bmatrix}
\end{equation}
By a change of basis we get coordinates that are standard 
in describing space-time in physics
(index ``ST'' may stand for both ``standard'' and ``space-time''):
\begin{equation}
\label{eq:xyzt}
\begin{bmatrix} \dot x\\ \dot y\\ z\\ t \end{bmatrix}_{\sf ST}
\quad  = \quad
\begin{bmatrix} \dot x\\ \dot y\\ (A-A^c)/2\\ (A+A^c)/2 \end{bmatrix}_{\sf ST}
\qquad 
G_{\sf ST} = \begin{bmatrix} 
-1 & 0 & 0 & 0\\
0 &-1 & 0 & 0\\
0 & 0 & -1 & 0\\
0 & 0 & 0 & 1
\end{bmatrix}
\end{equation}

\item
An arrangement of disks is {\bf integral} 
if the curvatures of the disks are integers.
It is called {\bf primitive} if additionally 
their greatest common divisor is 1.
It is {\bf superintegral} if the symbols \eqref{eq:symbol1} are integral
and {\bf hyperintegral} if also tangency spinors (see below) are integral.

\item
A {\bf Descartes configuration} is a system of four mutually tangent disks.
Descartes' formula relates the curvatures:
\begin{equation}
\label{eq:Descartes}
(A+B+C+D)^2 = 2(A^2 + B^2 + C^2 + D^2)
\end{equation}
It follows \cite{Soddy} that, given $A$, $B$, $C$, one has two solutions for $D$:
\begin{equation}
\label{eq:Descartes2}
D_{1,2} \ = \ A+B+C \pm 2\sqrt{(AB+BC+ CA}
\end{equation}
with $D_1$ and $D_2$ called Descartes {\bf conjugates} through $A,B,C$.

\item
A {\bf tricycle} is the system of three mutually tangent disks.
A {\bf Descartes completion} of a tricycle is the Descartes configuration containing the tricycle
as a subset (one of the two possible). 
A tricycle is {\bf proper} if it is primitive and its Descartes completion is integral.
Triple $(1,1,1)$ is integral yet not proper, since the fourth disk 
of its Descartes completion is  
$$
D \ = \  1+1+1 \pm \sqrt{1\!\cdot\! 1 \!+\! 1\!\cdot\! 1 \!+\! 1\!\cdot\! 1} = 3\pm 2\sqrt{3}\,,
$$
which is evidently non-rational.  
We also define the weight of the tricycle $t$ of curvatures $(A,B,C)$ as
$$
         w(t) \ = \ A+B+C 
$$
Note that any tricycle has positive weight.
The smallest possible  weight of an integral tricycle  is 1,
assumed by one presented in Figure \ref{fig:tricycles}. 

\item
The {\bf Apollonian completion} of a tricycle or Descartes configuration 
is the unique Apollonian disk packing 
containing it.
 \item
Apollonian completion of any integral (primitive) Descartes configuration is 
an integral (primitive) Apollonian disk packing.
This is true also for {\bf proper} (primitive) tricycles.

\item
A {\bf tangency spinor} is defined for any two tangent disks $A$ and $B$ as follows. 
Consider the plane as the complex plane
and the vector joining the center of $A$ to the center of $B$ 
as a complex number $z\in \mathbb C$.
The tangency spinor oriented from $A$ to $B$ is defined up to the sign as
\begin{equation}
\label{eq:spindef}
\spin(A,B) \ = \ \pm\sqrt{\frac{z}{r_A\,r_B}\;}
\end{equation}
where $r_A=1$ and $r_B$ are radii of the disks (possibly negative).
We will view it interchangeably either as a complex number or as a vector in $\mathbb R^2$:
\begin{equation}
\label{eq:spinnotation}
u=\tspin(A,B) = m+ni = \begin{bmatrix}
                   m \\ 
                   n
\end{bmatrix} \,,
\quad
u^+ \! =\tspin(B,A) =  -n+mi = \begin{bmatrix}
                   -n \\ 
                   m
\end{bmatrix} 
\end{equation}
(both up to a sign).
Using circle symbols
and following our convention of using the same symbol for a disk and its curvature, 
one may easily derive the formula
$$
u^2 \ = \ 
           \begin{vmatrix}
            A & B\\
            \dot x_A & \dot x_B
            \end{vmatrix}
          \  + \ 
            \begin{vmatrix}
            A & B\\
            \dot y_A & \dot y_B
            \end{vmatrix}
            \; i
$$
Or, equivalently, denoting 
$$
\Delta_{AB}=\left|\begin{matrix}   A  & B \\
                                               \dot x_A & \dot x_B  \end{matrix}\right|
\qquad
s \ = \ \hbox{sign of\  } \left(\;\left|\begin{matrix}   A  & B \\
                                                                             \dot y_A & \dot y_B \end{matrix}\right|\; \right)
$$
the spinor oriented from $A$ to $B$ can be calculated as
\begin{equation}
\label{eq:spincalc}
\spin(A,B) \ =  \ \sqrt{\dfrac{A+B+\Delta_{AB}}{2}}  
                           \; + \;
                           s\cdot\sqrt{\dfrac{A+B-\Delta_{AB}}{2}}\, i
\end{equation}
In any Apollonian disk packing, if any  two adjacent spinors,
i.e., spinors $\spin(A,B)$ and $\spin(A,C)$ for some tricycle in the packing,
are integral then all spinors in the packing are integral.
For more see \cite{jk-spinors,jk-Clifford,jk-corona,jk-lattices}.        
\end{itemize}

\section{Hyper-integrality --- the main result}
\label{sec:main}
Here is the main theorem:

\begin{theorem}
\label{thm:main}
\sf
For every irreducible integral Descartes configuration (and consequently Apollonian disk packing)
there exists a positioning in the Euclidean plane 
such that the following are integral:
\begin{enumerate}
\item
all entries of the disk symbols: curvatures (by assumption), co-curvatures, and reduced coordinates
$A, A^c, \dot x, \dot y$; \\[-19pt] 
\item
tangency spinors  $u=[m,n]^T$;  \\[-19pt]

\item 
disk vectors in the Minkowski space in the standard space-time coordinates,
i.e., $z$ and $t$ of \eqref{eq:xyzt}. 
\end{enumerate} 
\end{theorem}

\begin{proof}
Here is the outline of the proof, 
the details of which are clarified in the following section.
Start with an integral irreducible packing, understood as a structured set of curvatures
(positions are not necessarily integral, or not even known).
Pick an arbitrary tricycle (or Descartes configuration) in it, $A,B,C$.
For convenience, assume an increasing order $A\leq B \leq C$.
Perform a chain of the transformations $(A,B,C) \mapsto (A',B',C')=T(A,B,C)$
defined as a map
\begin{equation}
\label{eq:T}
T(A,B,C) = \begin{cases}
                (-A,\, B\!+\! 2A,\, C\!+\!2A) & \hbox{if}\   A<0 \qquad \hbox{(*)}\\
                (A,\,  B,\,  A\!+\!B\!+\!C\!-2\sqrt{AB\!+\! BC\! +\! CA}) & \hbox{otherwise}  \quad \hbox{(**)}
                \end{cases}
\end{equation}
followed by ordering the resulting triple to $A'\leq B' \leq C'$. 
Map (*) will be called  a ``descending self-inversion'' and is due to 
an inversion in the circle that bounds a disk of negative curvature $A$. 
Map (**) will be called the ``descending Descartes move'' and is due to
replacement of the smallest disk 
by the greater from the two that would complete the triple to Descartes configuration.
Both are described in detail in subsequent material.

In both cases the weight is strictly decreasing,
$$
         w(A',B',C') < w(A,B,C)\,,
$$
and preserves integrality of all ingredients (symbols and spinors).
Hence, the iteration of map $T$ must end up with the lowest-weight triple $(0,0,1)$.
Now, this particular triple may be realized geometrically as shown in Figure~\ref{fig:bottom} below,
with all entries of the symbol and of the tangency spinors being integral.

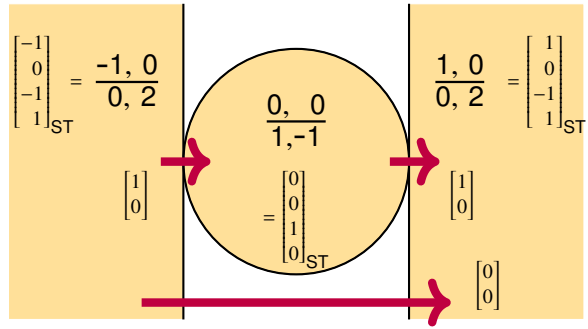
\begin{figure}[H]
\centering
\begin{tikzpicture}[scale=1.5]  
\draw  [fill=gold!40] (0,0) circle (1);
\draw  [color=white, fill=gold!40] (-2.55,-1.4) rectangle  (-1,1.4);
\draw  [color=white, fill=gold!40] (2.55,-1.4) rectangle  (1,1.4);
\draw [thick] (1,-1.4) -- (1,1.4)
       node  [scale=1.4]  at (1.45,.7) {$\frac{\sf 1,\  0}{\sf 0,\ 2}$}
       node  [scale=.7]  at (2.2, 0.7) {$\ =\begin{bmatrix}\ \ 1\\ \ \ 0\\-1\\ \ \ 1
                                                       \end{bmatrix}_{\hbox{\small \sf ST}}$};
\draw [thick] (-1,-1.4) -- (-1,1.4)
       node  [scale=1.4]  at (2.2 -3.7,.7) {$\frac{ \sf \z 1,\  0}{\sf \ 0,\ 2}$}
       node  [scale=.7]  at (1.5-3.7, 0.7) {$\begin{bmatrix*}-1\\ \ \ 0\\ -1\\ \ \ 1
                                                              \end{bmatrix*}_{\hbox{\small\sf ST}}\!\!\!=$};
\draw [thick] (0,0) circle (1)
          node  [scale=1.4]  at (0,.35) {$\frac{\sf 0,\  \  0}{\sf 1,\z 1}$}
         node  [scale=.7]  at (0,-0.5) {$=\begin{bmatrix}0\\0\\1\\0
                                                          \end{bmatrix}_{\hbox{\small \sf ST}}$};
\draw [->, line width=3.2pt, purple]  (.83,0) -- (1.27,0);
\node at (1.47, -.3) [scale=.7] {$\begin{bmatrix} 1 \\  0 \end{bmatrix}$};
\draw [->, line width=3.2pt, purple]  (-1.27+.07, 0) -- (-.83+.07,0);
\node at (-1.42, -.3) [scale=.7] {$\begin{bmatrix} 1 \\  0 \end{bmatrix}$};
\draw [->, line width=3.2pt, purple]  (-1.37, -1.25) -- (1.37, -1.25);
\node at (1.7, -1.1) [scale=.7] {$\begin{bmatrix} 0 \\ 0 \end{bmatrix}$};
\end{tikzpicture}
\caption{The lowest-weight base tricycle $\bar{\mathbf t}_0$in the integral position.}
\label{fig:bottom}
\end{figure}

Applying the inverse of the chain of the transformation to this tricycle
will recreate the original tricycle, but now  in the integral position.
The completion of this tricycle to the Apollonian packing
will result in the hyper-integral Apollonian packing 
with the curvatures that of the original packing.
The details are clarified in the subsequent section.
\end{proof}

\begin{figure}
\centering
\begin{tikzpicture}[auto,node distance=.5cm]
        \node  at (0, 5) {
        \begin{tikzpicture}[scale=5,rotate=77]  
\draw [thick] (3/6,4/6) circle (1/6);
\draw [fill=green!70, thick] (27/42,28/42) circle (1/42);
\draw [fill=green!70, thick] (14/23,14/23) circle (1/23);
\draw [fill=green!70, thick] (70/107, 68/107) circle (1/107);
\foreach \a/\b/\c/\d/\s in {
 6/ 8/ 11/ 94/2,
 7/ 8/ 14/ 8/2,
 6/ 10/ 15/ 9/2,
14/ 14/ 23/ 17/1,
11/ 20/ 26/ 20/1, 14/ 20/ 35/ 17/.7,  27/ 28/ 42/ 36/.7,
22/ 38/ 47/ 41/.7, 30/ 28/ 51/ 33/.7, 22/ 44/ 59/ 41/.7
}
\draw (\a/\c,\b/\c) circle (1/\c)
;
\foreach \a/\b/\c/\d in {
30/ 38/ 71/ 33,  46/ 50/ 71/ 65, 27/ 44/ 74/ 36,
39/ 64/ 78/ 72,   41/ 56/ 86/ 56, 54/ 50/ 95/ 57,
57/ 64/ 102/ 72, 70/ 68/ 107/ 89,  49/ 80/ 110/ 80,
39/ 80/ 110/ 72,  71/ 80/ 110/ 104,  62/ 98/ 119/ 113,
75/ 68/ 122/ 84,  54/ 64/ 123/ 57,  54/ 100/ 123/ 105,
46/ 80/ 131/ 65,  57/ 80/ 134/ 72,  54/ 110/ 143/ 105,
86/ 80/ 155/ 89,  97/ 104/ 158/ 128,  102/ 118/ 159/ 153,
91/ 140/ 170/ 164,  81/ 136/ 174/ 144, 62/ 128/ 179/ 113,
105/ 104/ 182/ 120, 75/ 100/ 186/ 84,  86/ 98/ 191/ 89,
70/ 110/ 191/ 89
}
\draw (\a/\c,\b/\c) circle (1/\c)  
;
\end{tikzpicture}
 };

        \node at (10,5) {
        \begin{tikzpicture}[scale=5]  
\draw [thick] (3/6,4/6) circle (1/6);
\draw [fill=gold!70, thick] (27/42,28/42) circle (1/42);
\draw [fill=gold!70, thick] (14/23,14/23) circle (1/23);
\draw [fill=gold!70, thick] (70/107, 68/107) circle (1/107);
\foreach \a/\b/\c/\d in {
 6/ 8/ 11/ 94,    7/ 8/ 14/ 8,   6/ 10/ 15/ 9,
14/ 14/ 23/ 17,  11/ 20/ 26/ 20,  14/ 20/ 35/ 17,
27/ 28/ 42/ 36,  22/ 38/ 47/ 41,  30/ 28/ 51/ 33,
22/ 44/ 59/ 41,  30/ 38/ 71/ 33,  46/ 50/ 71/ 65,
27/ 44/ 74/ 36,  39/ 64/ 78/ 72,  41/ 56/ 86/ 56,
54/ 50/ 95/ 57,  57/ 64/ 102/ 72,  70/ 68/ 107/ 89,
49/ 80/ 110/ 80,  39/ 80/ 110/ 72,  71/ 80/ 110/ 104,
62/ 98/ 119/ 113,  75/ 68/ 122/ 84,  54/ 64/ 123/ 57,
54/ 100/ 123/ 105,  46/ 80/ 131/ 65,  57/ 80/ 134/ 72,
54/ 110/ 143/ 105,  86/ 80/ 155/ 89,  97/ 104/ 158/ 128,
102/ 118/ 159/ 153,  91/ 140/ 170/ 164,  81/ 136/ 174/ 144,
62/ 128/ 179/ 113,  105/ 104/ 182/ 120,  75/ 100/ 186/ 84,
86/ 98/ 191/ 89,  70/ 110/ 191/ 89
}
\draw (\a/\c,\b/\c) circle (1/\c)  
;

\draw [blue,->, line width=.4mm] (3/5+.012+.004, 4/5+.015+.005)--(3/5-.012-.004, 4/5-.015-.005);
\draw [blue,->, line width=.4mm] (4/8, 4/8-.028)--(4/8, 4/8+.028);
\draw [blue,->, line width=.4mm] (3/9-.02, 6/9)--(3/9+.02, 6/9);
\draw [red,->, line width=.4mm] (12/26 -.02, 18/26-.01)--(12/26+.02, 18/26+.01);
\draw [red,->, line width=.4mm] (13/25 +.01, 16/25+.02)--(13/25-.01, 16/25-.02);
\draw [red,->, line width=.4mm] (13/29 +.012, 18/29-.015)--(13/29-.012, 18/29+.015);
\end{tikzpicture}
};        
                
        \node  at (3, 5) {
        \begin{tikzpicture}[scale=5,rotate=77]  
\draw [fill=green!70] (27/42,28/42) circle (1/42);
\draw [fill=green!70] (14/23,14/23) circle (1/23);
\draw [fill=green!70] (70/107, 68/107) circle (1/107);
\draw [thick] (27/42,28/42) circle (1/42);
\draw [thick] (14/23,14/23) circle (1/23);
\draw [thick] (70/107, 68/107) circle (1/107);
\end{tikzpicture}
 };

        \node  at (3, 4.2) {$\sf (A,B,C)$};
        \node  at (7, 4.2) {$\sf (A',B',C')$};                
                
        \node  at (2.7, 1) {$\bf t_0 = \sf (0,0,1)$};          
        \node  at (8, 1) {$ = \bar{\bf t}_0 $};          
        
        \node  at (7, 5) {
        \begin{tikzpicture}[scale=5]  
\draw [fill=gold!70, thick] (27/42,28/42) circle (1/42);
\draw [fill=gold!70, thick] (14/23,14/23) circle (1/23);
\draw [fill=gold!70, thick] (70/107, 68/107) circle (1/107);
\draw [thick] (27/42,28/42) circle (1/42);
\draw [thick] (14/23,14/23) circle (1/23);
\draw [thick] (70/107, 68/107) circle (1/107);
\end{tikzpicture}
};                                                

        \node  at (7, 1) {
        \begin{tikzpicture}[scale=.3]  
\draw [fill=gold!70] (0,0) circle (1);
\draw [gold!20, fill=gold!25] (-1.6,-1.7) rectangle (-1,1.7);
\draw [gold!20, fill=gold!25] (1, -1.7) rectangle (1.6, 1.7);
\draw [gold!20, fill=gold!35] (-1.4,-1.7) rectangle (-1,1.7);
\draw [gold!20, fill=gold!35] (1, -1.7) rectangle (1.4, 1.7);
      
\draw [black, thick] (0,0) circle (1);
\draw [black, very thick] (-1,-1.7) -- (-1,1.7);
\draw [black, very thick] (1, -1.7) -- (1, 1.7);

\end{tikzpicture}
 };                                                
                
\draw [->, line width=.5mm] (3,3.8) -- (3,1.6) ;  \node at (3.2, 2.7) {$\sf T$};
\draw [<-, line width=.5mm] (7,4) -- (7,1.7);   \node at (7.4, 2.7) {$\sf T^{-1}$};     
\draw [->, line width=.5mm] (4,1) -- (6.2,1);      \node at (5, 1.2) [scale=.7] {\sf geometrization};            


\draw [->,  line width=.5mm] (1,5.4) to [bend left=30] (2.5,5.4);  \node at (1.8,5.9) {\sf pick};
\draw [->,  line width=.5mm] (7.5,5.4) to [bend left=30] (9,5.4);  \node at (8,5.9) {\sf completion};

\node at (0,3.2) [text width=2 cm] {\sf integral Apollonian packing};
\node at (10.2,3.2) [text width=2.2 cm] {\sf hyperintegral Apollonian packing};
                
    \end{tikzpicture}  

\caption{The idea of the proof.  Symbol $T$ stands 
for the composition of the descending transformations, 
and $T^{-1}$ for its inverse. Triple $(A,B,C)$ are curvatures of an arbitrary selection of a tricycle,
and $(A',B',C'$) the tricycle with symbols and spinors being integral.}
\label{fig:idea}
\end{figure}
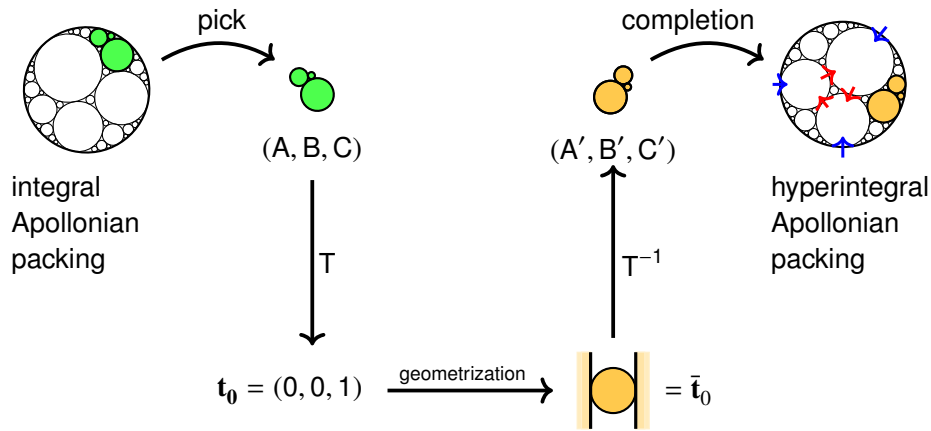

The puzzling property of Proposition \ref{thm:sum1} 
receives an indirect proof:

\begin{corollary}
\label{thm:sum-done}
\sf
Since the tangency spinors in any integral primitive Apollonian packing may be made integral,
and its norm squared is the sum of the corresponding disks,
Proposition \ref{thm:sum1} holds.
\end{corollary}

\begin{corollary}
\sf 
In an irreducible Apollonian disk packing, the curvatures and co-curvatures of each disk 
share the parity, i.e., $2|(A+A^c)$. 
This due to the fact that $t$ and $z$ in the standard basis \eqref{eq:xyzt}
admit integrality too.
\end{corollary}

\begin{figure}[H]
\centering
\begin{tikzpicture}[scale=2.1]  
\draw  [color=white, fill=gold!40] (-1.42,-1.2) rectangle  (1.42,1.34);
\draw  [fill=white] (0,0) circle (1);
\draw  [fill=gold!40] (1/2,0) circle  (1/2);
\draw  [fill=gold!40] (-1/2,0) circle  (1/2);
\draw  
       node  [scale=1.34]  at (1/2,0) {$\frac{\sf 1,\  0}{\sf 2,\ 0}$};
\draw
       node  [scale=1.34]  at (-1/2,0) {$\frac{\sf -1,\  0 \  }{\sf \ 2,\ 0}$};
\draw [thick] (0,0) circle (1)
          node  [scale=1.34]  at (-.92,.91) {$\frac{\sf \ \ 0,\  0 \  }{\sf \z 1,\ 1}$};

\draw [->, line width=3pt, purple]  (.83,0) -- (1.2,0);
\node at (1.22, -.3) [scale=.7] {$\begin{bmatrix} 1 \\  0 \end{bmatrix}$};
\draw [->, line width=3pt, purple]  (-1.19, 0) -- (-.79,0);
\node at (-1.17, -.3) [scale=.7] {$\begin{bmatrix} 1 \\  0 \end{bmatrix}$};
\draw [->, line width=3pt, purple]  (-.15, 0) -- (.2, 0);
\node at (0, -.4) [scale=.7] {$\begin{bmatrix} 2 \\ 0 \end{bmatrix}$};
\end{tikzpicture}

\caption{A tricycle generating the Apollonian Window, 
the image of inversion of which is the lowest weight tricycle, see Fig.~\ref{fig:bottom}.}
\label{fig:bottom1}
\end{figure}
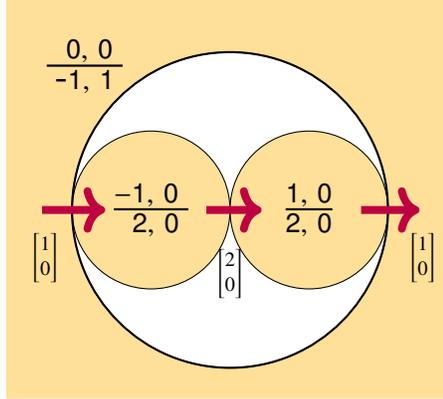

\begin{remark}
The symbols and spinors of the base tricycle shown in Figure~\ref{fig:bottom}
are slightly challenging due to the presence of half-plane disks.  
Hence, it is convenient to 
start with a clear tricycle shown in Figure \ref{fig:bottom1} 
(known from the Apollonian Window) and apply inversion through the main circle $(0,0)/(-1,1)$.
\end{remark}

\section{Tricycle walks}

In this section we describe the two types of moves of \eqref{eq:T}
in the space of tricycles.

\subsection{Tricycle self-inversions}  
\label{sec:si}

Every disk $X$ defines a circle $\del X$, its boundary.
As such, it defines an inversion through this circle, $\Inv_{\del X}$.
It can be shown that  
inversion of disk $C$ in circle $K$
corresponds in the Minkowski space to reflection of vector $\mathbf C$
in the hyper-plane $\mathbf K^\bot$ orthogonal to vector $\mathbf K$. 
The well-known formula for such reflection in the Euclidean or pseudo-Euclidean space is
$$
\mathbf C \quad  \mapsto \quad \Inv_K(\mathbf C) = 
\mathbf C - 2\frac{\langle \mathbf K, \mathbf C\rangle}{\langle \mathbf K, \mathbf K\rangle} \mathbf K
$$
Since the norm squared of disks is $(-1)$, the formula for disks reduces to
$$
\mathbf C\quad \mapsto \quad \Inv_K(\mathbf C) =
\mathbf C + 2\langle \mathbf K, \mathbf C\rangle\,  \mathbf K
$$
In addition, if disk $C$ and $K$ are tangent, the formula reduces further to
\begin{equation}
\label{eq:invAformula}
\Inv_K: \ \mathbf C\ \mapsto \ 
\mathbf C + 2 \mathbf K
\end{equation}
In the symbol-like notation:
\begin{equation}
\label{eq:invformula}
\frac{\dot x_C,\, \dot y_C}{C, C^c}\ \mapsto \
\frac{\dot x_C+2 \dot x_K,\, \dot y_C+2\dot x_K}{C+2K,\; C^c+2K^c}
\end{equation}
If  both disks are  integral, so is the image $C'$.

\begin{definition}
{\bf Self-inversion} of a tricycle $(A,B,C)$ is an inversion through 
one of its circles $\del A$, $\del B$, or $\del C$.
{\bf A descending self-inversion} is defined as 
the inversion through the circle bounding the disk of negative curvature 
if such exists among $(A,B,C)$.  Otherwise, it is not defined.
\end{definition}

Following \eqref{eq:invformula}, the curvatures of a tricycle are transformed 
by self-inversion through $\del A$ as follows
\begin{equation}
\label{eq:2A}
(A,B,C) \quad \xrightarrow{~~\Inv_{\del A} ~~} \quad (-A, B+2A, C+2A)
\end{equation}
Assuming the first disk being negative, the above describes the descending self-inversion. 
\\
\\
{\bf Remark:} The self-inversive moves define a 3-valent graph, and the {\it descending} self-inversive move
defines a digraph based on the same graph.

\begin{proposition} 
\sf
Under a descending self-inversive transformation, the weight function is strictly decreasing.
\end{proposition}

\begin{proof}
A simple verification: 
denoting $\mathbf t =(A,B,C)$ and $\mathbf t'=T(\mathbf t)=( A',B', C')$, we have
$$
w(\mathbf t')  = A'+B'+C' = (-A) + (B \!+\!2A) + (C\!+\!2A) = w\left( T(\mathbf t)\right) + 2A
$$
Since descending self-inversion is defined for $A<0$, the claim holds.
\end{proof}

\begin{corollary}
\sf 
Eq. \eqref{eq:2A} implies that the image of a proper tricycle through a self-inversion is proper.
\end{corollary}

\begin{figure}
\centering
\includegraphics[scale=.57]{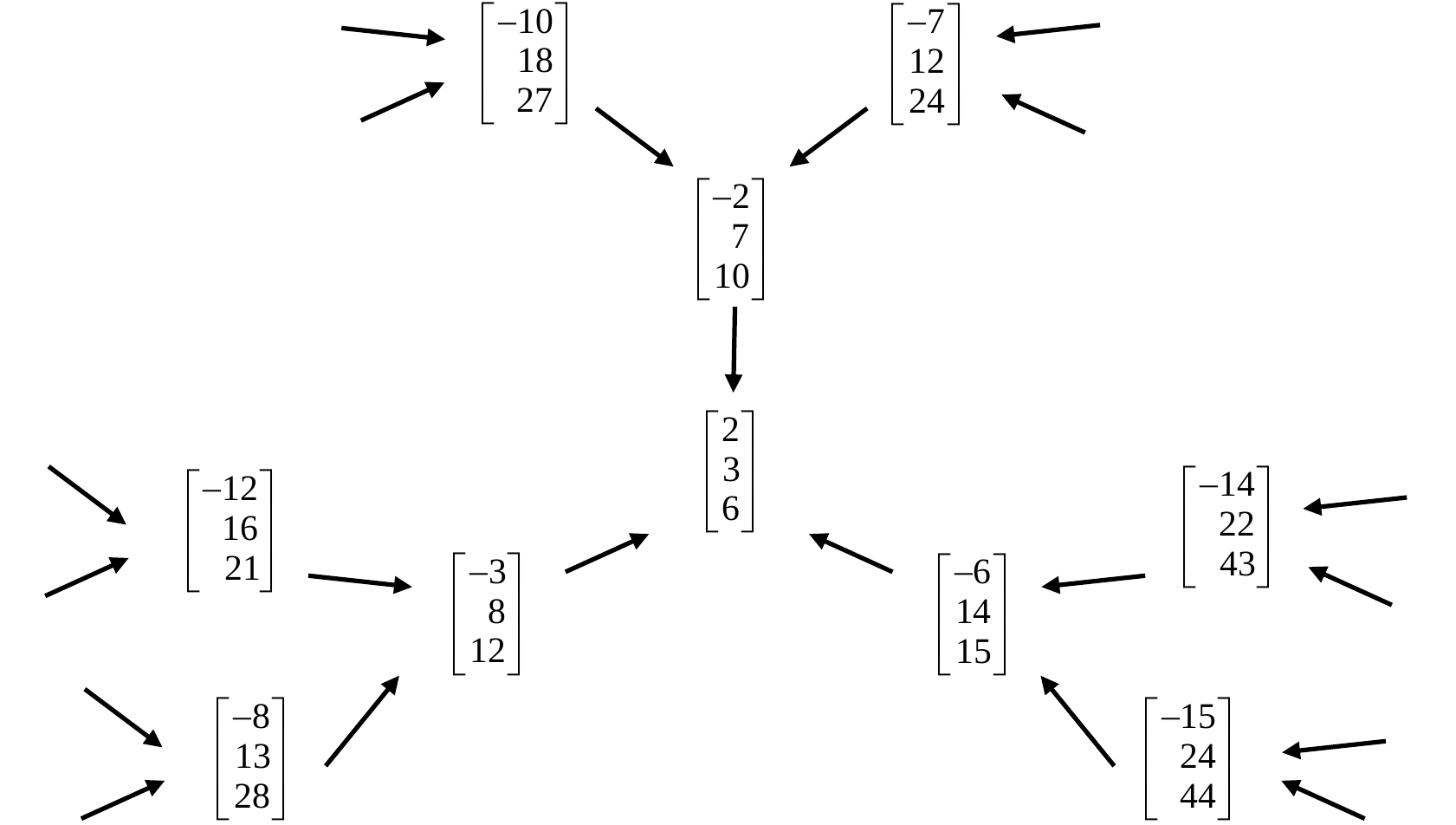}
\caption{The digraph of tricycles with the relation defined by descending self-inversion.
Note that $[2,3,6]$ is the terminal vertex.}
\label{fig:236}
\end{figure}

\begin{proposition}
\sf 
Under self-inversion, spinors preserve the integrality.
\end{proposition}

\begin{proof}
Given mutually tangent disks $A$ and $B$ of symbols (see Figure~\ref{fig:spinInverted})
$$
\mathbf A = \frac{\dot x_A, \; \dot y_A}{A,\; A^*} 
\qquad\hbox{and}\qquad
\mathbf B = \frac{\dot x_B, \; \dot y_B}{B,\; B^*} \;,
$$
the self-inversion through $A$ results in disks of symbols
$$
\mathbf A' =\hbox{Inv}_A(\mathbf A) =  \frac{-\dot x_A, \; -\dot y_A}{-A} 
\quad\hbox{and}\quad
\mathbf B' =\hbox{Inv}_A(\mathbf B) =  \frac{\dot x_B\!+\!2\dot x_A, \; \dot y_B\!+\!2\dot y_A}{B+2A} \,.
$$ 
Denoting $u=\spin(A,B)$ and $u'=\spin(B',A')$,   
the square of the spinor from $A$ to $B$, understood as  a complex number, is
\begin{equation}
\label{eq:uABbefore}
u^2 \ = \ 
           \begin{vmatrix}
            A & B\\
            \dot x_A & \dot x_B
            \end{vmatrix}
          \  + \ 
            \begin{vmatrix}
            A & B\\
            \dot y_A & \dot y_B
            \end{vmatrix}
            \; i
\end{equation}
while the square of the spinor $\spin(A', B')$ under inversion through $\partial A$ is  
\begin{equation}
\label{eq:uABafter}
u'^2 =  \begin{vmatrix}
            -A & B+2A\\
            -\dot x_A & \dot x_B +2\dot x_A
            \end{vmatrix}
            +
            \begin{vmatrix}
            -A & B+2A\\
            -\dot y_A & \dot y_B+2\dot y_A
            \end{vmatrix}
            \; i\,.
\end{equation}
The determinants in \eqref{eq:uABafter} are negatives of the determinants 
of \eqref{eq:uABbefore}.
Indeed, add the first column twice to the second, and factor out the sign.
Thus $u'^2=-u^2$,
or $u'=\pm iu$.
If 
$u = [m,n]$ then under the inversion we get $u' = [m, -n ]$ (up to the sign, of course).
Thus, if $u$ is integral, so is $u'$. 
\end{proof}

In conclusion, the spin under inversion becomes the conjugated vector, $\mathbf u^+$:
$$
\tspin(A',B')^2 = u^2  \ = \  -\tspin(A,B)^2\,.
$$

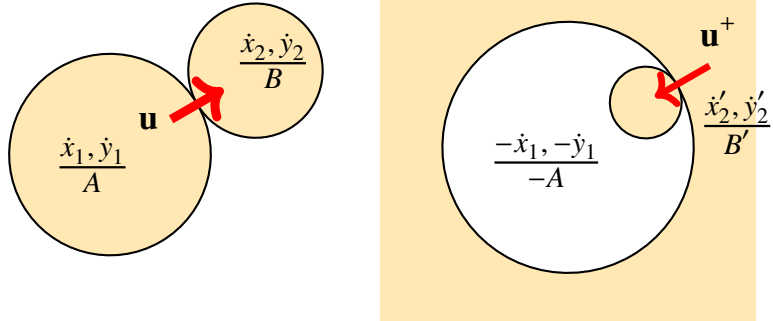
\begin{figure}[h]
\centering
\begin{tikzpicture}[scale=1.34, rotate=30]
\draw [fill=gold!30, thick] (0, 0) circle (1);
\draw [fill=gold!30, thick] (5/3,0) circle (2/3);
\node [scale=1.4] at (-.2,0) {$\frac{\dot x_1,\, \dot y_1}{A}$};
\node [scale=1.4] at (5/3+.2,  0) {$\frac{\dot x_2,\, \dot y_2}{B}$};
\draw [red,->, line width=1.2mm] (1-.3, 0)--(1+.3, 0);
\node  [scale=1.2]  at (.5, .1)  {$\mathbf u$};
\draw [fill=white,color=white] (0,-1.8) circle (.1);
\end{tikzpicture}
\; \;\;
\begin{tikzpicture}[scale=1.67, rotate=30]
\draw [color=white, fill=gold!30, rotate=-30] (-1.5,-1.4) rectangle (1.5,1.4); 
\draw [fill=white, thick] (0, 0) circle (1);
\draw [fill=gold!30, thick] (5/7,0) circle (2/7);
\node [scale=1.4] at (-.2,0) {$\frac{-\dot x_1, \, -\dot y_1}{-A}$};
\node [scale=1.4] at (5/3-.4,  -.5) {$\frac{\dot x'_2, \,\dot y'_2}{B'}$};
\draw [red,->, line width=.9mm] (1+.3, 0)--(1-.2, 0);
\node  [scale=1.2]  at (1.5, 0.2)  {$\mathbf u^+$};
\end{tikzpicture}

\caption{Spinor under inversion}
\label{fig:spinInverted}
\end{figure}

\subsection{Descartes move of tricycles}

Now we inspect the other move of \eqref{eq:T},  
which we shall call {\bf Descartes move}.
It consists of adding  one of the two disks that complete a tricycle to a Descartes configuration 
and simultaneously removing one of its original three disks.
For curvatures,
\begin{equation}
\label{eq:abcsqrt}
            (A,B,C) \ \to \  (A, \, B, \; A+B+C \pm 2\sqrt{AB+BC+CA})
\end{equation}
The move can be done in 6 different ways, hence every tricycle 
determines a 6-valent graph as its orbit via such moves.
The vertices of this graph are all tricycles that appear in a given Apollonian disk packing.
It is connected --- 
any two tricycles can be connected by a chain of Descartes moves.
A {\bf Descartes walk} is a path in this graph.

\begin{definition}
A {\bf descending Descartes move} is a Descartes move in which the new disk 
is the greater among the two completing disks,
and the removed disk is the one of the greatest curvature.   
\end{definition}

\noindent
{\bf Remark:}
If one follows the convention of ordering the curvatures $A\leq B \leq C$, 
the descending Descartes move corresponds to \eqref{eq:abcsqrt}.
It may be expressed algebraically 
$$
            (A,B,C) \ \to \  
          \left(m, \, \Sigma\! -\!M\!-\!m, \, 
            \Sigma -\!2\sqrt{AB\!+\!BC\!+\!CA}\right)
$$             
where $M=\max(A,B,C)$, $m=\min(A,B,C)$, and $\Sigma = A+B+C$.

~

Here is an example of two Descartes  descending moves 
$$
(2,3,23) \ \to \ (2,3,6)   \ \to \ (-1,2,3)
$$

\begin{proposition}
\sf 
The descending Descartes move applied to a bounded tricycle
(tricycle with non-negative curvatures)
decreases strictly the weight, with a single exception of tricycle $(0,0,1)$
\end{proposition}

\begin{proof}
Assume the order of curvatures to be $0\leq A\leq B \leq C$.
We need to find out if the curvature $C'$ of the new disk replacing the smallest disk $C$ 
is indeed greater than $C$ (has smaller curvature).
Here is a chain of algebraic implications:
\begin{equation}
\label{eq:proof}
\begin{array}{rcl}
      &\quad&C ' \ \leq \ C \\[5pt]
\Liff && A+B+C-2\sqrt{AB+BC+CA} \ \leq \ C  \\[5pt]
\Liff && A+B\  \leq \ 2\sqrt{AB+BC+CA}        \\[5pt]
\Liff && 0  \  \leq \ 2AB+4BC+4CA-A^2 -B^2        \\[5pt]
\Lthenif && 0 \ \leq \ 3A^2 + 2AB + 3B^2   
\end{array}
\end{equation}
where in the last expression $C$ is replaced by smaller values, $A$ and $B$.
The last inequality is self-evident, as the right-hand side is a sum of squares:
$(A+B)^2 +2A^2 +2B^2$.        
The only situation when the equality $C=C'$ holds 
is when $A=B=0$.
The primitivity of the triple enforces $C=1$ 
resulting in the tricycle $(0,0,1)$.
\end{proof}

\begin{proposition}
\sf
The Descartes move preserves the integrality of the disk symbols and of the tangency spinors.
\end{proposition}

\begin{proof}
The integrality of symbols is  
due to the general rule that 
 the conjugated disks $D$ and $D'$ (two solutions to Descartes formula),
satisfy
$$
\omega_4+ \omega'_4 \ = \  2\omega_1 + 2\omega_2 + 2\omega_3
$$
for $\omega$ denoting collectively any of the coefficients, 
$\dot x$, or $\dot y$, or curvature, or co-curvature, of the $i$-th disk.
Since a proper integral tricycle determines a super-integral Apollonian packing, 
and since the Descartes move travels along tricycles in such a packing,
super-integrality is preserved trivially.  Since the integrality of two adjacent spinors 
in any Descartes configuration determines integrality of all its spinors \cite{jk-lattices},
the claim hold trivially.
\end{proof}  

~

\noindent
{\bf Remark:} 
The Descartes descending move is the critical notion in setting a dynamical system
that leads to the definition of the Apollonian depth function in \cite{jk-depth,jk-universal}.

\subsection{Descending process}

The process may be alternating between 
unbounded triples (the greatest circle bounds a disk of negative curvature)
and bounded (all three are of positive curvatures).
The weight in the descending process is strictly decreasing until reaching the value $w=1$.

~

\noindent
{\bf Example:}
Here is an example of a proper tricycle $(11,14,86)$ from Figure \ref{fig:A2} undergoing a series of
descending moves:
{\small
\begin{equation}
\label{eq:left-right}
\begin{array}{cccccc} 
& (11, 14, 86) \!\!\! \!\!\\
&& \searrow\\
&&& (11, 14, 15) \!\!\! \\
&&&& \searrow\\
&&&&&\!\!\!\! (-6, 11, 14)  \\
&&&& \swarrow\\
&&&(-1, 2,  6)   \\
&&\swarrow\\
&(\,0,\,1,\, 4\,)\!\!\!   \\
&&\searrow\\
&&&(\,0,\,1,\,1\,)   \\
&&&&\searrow\\
&&&&& \!\!\!\!\!\!(\,0,\,0,\,1\,)
\end{array}
\end{equation}
}
The above chain is to be read from the top down.  
Arrows to the right denote Descartes moves, 
to the left --- self-inversions.

~

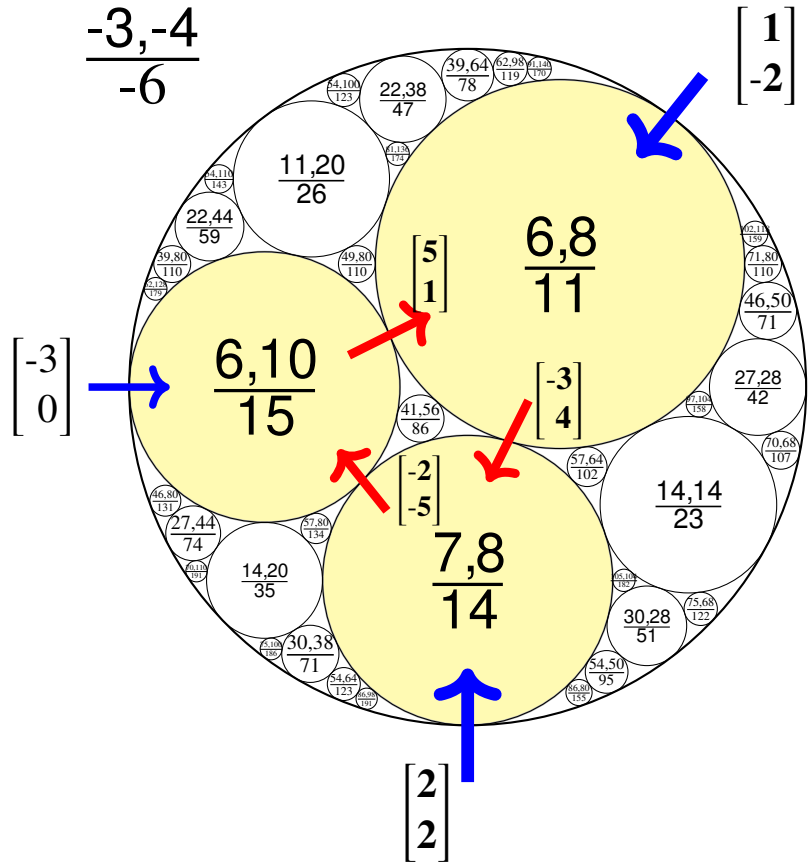
\begin{figure}[H]
\centering
\begin{tikzpicture}[scale=27]  
\draw [thick] (3/6,4/6) circle (1/6);
\draw [fill=yellow, opacity=.3, thick] (6/11,8/11) circle (1/11);
\draw [fill=yellow, opacity=.3, thick] (7/14,8/14) circle (1/14);
\draw [fill=yellow, opacity=.3, thick] (6/15,10/15) circle (1/15);
\foreach \a/\b/\c/\d/\s in {
 6/ 8/ 11/ 94/2,
 7/ 8/ 14/ 8/2,
 6/ 10/ 15/ 9/2,
14/ 14/ 23/ 17/1,
11/ 20/ 26/ 20/1,
14/ 20/ 35/ 17/.7,
27/ 28/ 42/ 36/.7,
22/ 38/ 47/ 41/.7,
30/ 28/ 51/ 33/.7,
22/ 44/ 59/ 41/.7
}
\draw (\a/\c,\b/\c) circle (1/\c)
        node [scale= \s]  at (\a/\c,\b/\c)  {$\frac{\hbox{\sf \a,\!\!\b\!\! }}{\!\!\hbox{\sf \c}}$}
;
\foreach \a/\b/\c/\d in {
30/ 38/ 71/ 33,
46/ 50/ 71/ 65,
27/ 44/ 74/ 36,
39/ 64/ 78/ 72,
41/ 56/ 86/ 56,
54/ 50/ 95/ 57,
57/ 64/ 102/ 72,
70/ 68/ 107/ 89,
49/ 80/ 110/ 80,
39/ 80/ 110/ 72,
71/ 80/ 110/ 104,
62/ 98/ 119/ 113,
75/ 68/ 122/ 84,
54/ 64/ 123/ 57,
54/ 100/ 123/ 105,
46/ 80/ 131/ 65,
57/ 80/ 134/ 72,
54/ 110/ 143/ 105,
86/ 80/ 155/ 89,
97/ 104/ 158/ 128,
102/ 118/ 159/ 153,
91/ 140/ 170/ 164,
81/ 136/ 174/ 144,
62/ 128/ 179/ 113,
105/ 104/ 182/ 120,
75/ 100/ 186/ 84,
86/ 98/ 191/ 89,
70/ 110/ 191/ 89
}
\draw (\a/\c,\b/\c) circle (1/\c)  
node [scale= 70/\c]  at (\a/\c,\b/\c)  {$\frac{\a,\b}{\c}$}
;

\node [scale= 2]  at (.34, .83)  {$\frac{\hbox{\sf -3,-4}}{\hbox{-6}}$};

\draw [blue,->, line width=1.7mm] (3/5+.012+.004, 4/5+.015+.005)--(3/5-.012-.004, 4/5-.015-.005);
\draw [blue,->, line width=1.7mm] (4/8, 4/8-.028)--(4/8, 4/8+.028);
\draw [blue,->, line width=1mm] (3/9-.02, 6/9)--(3/9+.02, 6/9);

\draw node  [scale=1.4]  at (3/5+.047, 4/5+.03) {\color{black} $\begin{bmatrix}\;\mathbf 1\\ \hbox{-}\mathbf 2\end{bmatrix}$};
\draw node  [scale=1.4]  at (4/8-.02, 4/8-.043) {\color{black} $\begin{bmatrix}\mathbf 2 \\ \mathbf 2\end{bmatrix}$};
\draw node  [scale=1.4]  at (3/9-.043, 6/9) {\color{black} $\begin{bmatrix}\hbox{-}3 \\ \;0\end{bmatrix}$};

\draw [red,->, line width=1.2mm] (12/26 -.02, 18/26-.01)--(12/26+.02, 18/26+.01);
\draw [red,->, line width=1.2mm] (13/25 +.01, 16/25+.02)--(13/25-.01, 16/25-.02);
\draw [red,->, line width=1.2mm] (13/29 +.012, 18/29-.015)--(13/29-.012, 18/29+.015);
\draw node  [scale=1.1]  at (12/26+.02,18/26+.03) {\color{black} $\begin{bmatrix}\mathbf 5 \\ \mathbf 1\end{bmatrix}$};
\draw node  [scale=1.1]  at (13/25 +.025, 16/25+.02) {\color{black} $\begin{bmatrix}\hbox{-}\mathbf 3 \\ \; \mathbf 4\end{bmatrix}$};
\draw node  [scale=1.0]  at (13/29 +.028, 18/29-.005) {\color{black} $\begin{bmatrix}\hbox{-}\mathbf 2 \\ \hbox{-}\mathbf 5\end{bmatrix}$};

\end{tikzpicture}

\caption{All data reconstructed for the triple $(11,14,15)$ from Figure \ref{fig:A2}}
\label{fig:all}
\end{figure}

The existential aim of the proof aside, 
the process may be used to actual construction of the ``hyper-integral'' 
version of any integral Apollonian disk packing.
The result of the reconstruction for the example shown 
in~\eqref{eq:left-right} is shown in Figure~\ref{fig:all}.  
For practical purposes, one must apply a labeling system to code 
to perform appropriate ``inverses''.

\section{The principal triples under descending transformations}

Every integral Apollonian packing has a triple of maximal size disks,
and the corresponding triple of the minimal curvatures.  
We shall call them {\bf principal tricycles} and {\bf principal triples}, respectively. 
Since such a triple determines the Apollonian packing, it may be viewed as its label.

Now, using the concept  of the descending transformation, 
we may construct a graph of descending paths for each of the principal triples.
Figure~\ref{fig:paths} shows the result
that includes triples 
with the negative curvature up to (-12). 
(A longer list of principal triples may be extracted from the list 
of Descartes quadruples in \cite{jk-Diophantine}
or produced with the algorithm discussed there.)
In the figure, the principal tricycles are in rectangles, the other in ovals.
From any starting tricycle, by descending moves one arrives in configuration $(0,0,1)$
where the paths meet.
Note that many paths go through the other  principal triples.
The branches will be called ``threads.''

~

\begin{figure}[h]
\centering
\includegraphics[scale=.62]{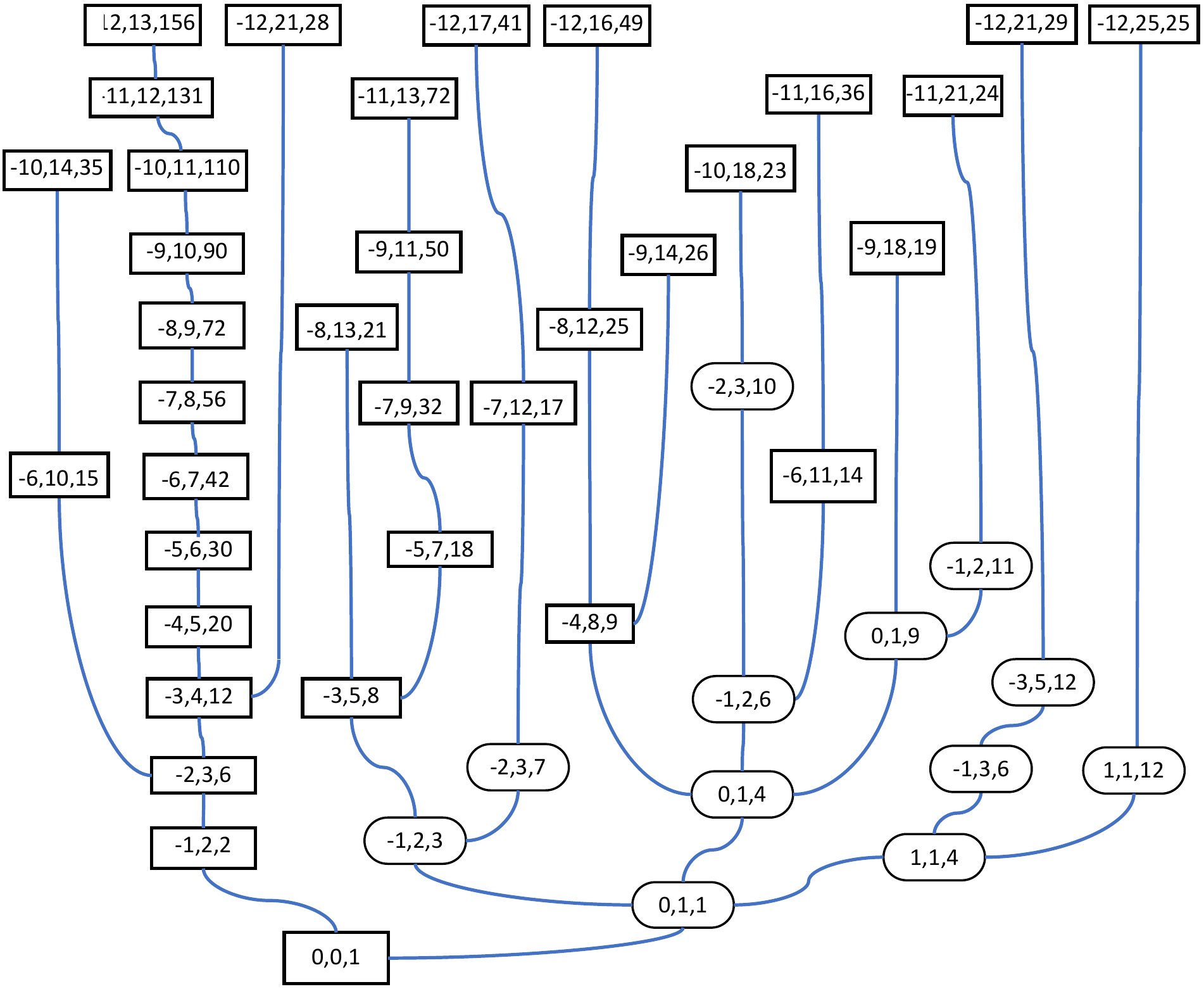}
\caption{Directed graph of primitive defining tricycles}
\label{fig:paths}
\end{figure}

An elementary inspection reveals that most of  
the paths form polynomial sequences.
Table  \ref{table} shows the polynomials for the paths in Figure \ref{fig:paths}.

\begin{table}
$$
\begin{array}{llccc}
\small \sf Label & \small\sf Example &\small\sf Formula && \small\sf Fourth \ disk \\[7pt]
A& (-12,13,156):  &   \Big(\,-n,\; n\!+\!1,\; n(n\!+\!1)\,\Big)                   
                           &&n^2\!+\!n\!+\!1 \\
B& (-11,13, 72):    &   \Big(\,-\!(2n\! -\! 1),\; 2n\! +\! 1,\; 2n^2\,\Big)                
                           &&     2n^2+2 \pm 2   \\
C& (-12, 17, 41):   &   \Big(\,-\!(5n\! -\! 3),\; 5n\! +\! 2, \;5n^2\! -\!n\! -\!1\,\Big)   
                            && 5n^2n\! +\! 4\pm2\\
D&  (-12, 16, 49):  & \Big(\, -4n,\; 4n\! +\! 4,\; (2n\! +\! 1)^2 \,\Big)   
                            && (2n\! +\! 1)^2\! +\! 4\pm 4\\
E&  (-11, 16, 36):   & \Big(\,-\!(5n\! -\! 4),\; 5n\! +\! 1,\; (5n\! -\! 3)n\,\Big)        
                            && 5n^2-n +5 \pm 4\\
F&  {\sf Fibonacci}                      & \Big(\,-F_{2n},\; F_{2n+1},\; F_{2n+2}\,\Big)           
                            &&  2F_{2n+1} \pm 2 \\
G& (-10,14, 35):    &   \Big(\,-\!(4n\! -\! 2),\; 4n\! +\! 2,\; 4n^2\! -\! 1\,\Big)                
                            && 4n^2+3       \\
\end{array}
$$
\vspace{-.22in}
\caption{The formulas for selected threads in the graph in Figure \ref{fig:paths}.}
\label{table}
\end{table}

There is however one surprise. 
One of the paths is not polynomial but instead consists of triples from the Fibonacci sequence. 
Note the negative the greatest polynomials are of degree1, while the third is quadratic.
We use the following convention for labeling the Fibonacci sequence:
$$
 F_0=0, \quad F_1=1, \quad F_{n+1}=F_n+F_{n-1}
 $$ 

\begin{proposition}
\label{thm:F}
\sf 
Every triple of the consecutive Fibonacci terms, starting with an even-labeled entry (made negative), 
$$
(\, -F_{2n},\; F_{2n+1}, \; F_{2n+2} \,) \,,
\qquad
n>1
$$
defines a principal triple of an Apollonian packing.
Moreover, the descending map (self-inversion) carries it to another triple of this kind:
\begin{equation}
\label{eq:F}
(\, -F_{2n},\; F_{2n+1}, \; F_{2n+2} \,)
\quad \mapsto\quad
(\, -F_{2n-2},\; F_{2n-1}, \; F_{2n} \,)\,.
\end{equation}
\end{proposition}

\begin{proof}
First, let us see that the triple is proper, that is, it completes to an integral Descartes configuration.
Using \eqref{eq:Descartes2}, the fourth disk curvature is
$$
D= -F_{2n}+F_{2n+1}+F_{2n} \pm2\sqrt{-F_{2n}F_{2n+1}+F_{2n+1}F_{2n+2}-F_{2n+2}F_{2n}}
$$
It takes only the basic identities to arrive at
$$
D = 2F_{2n+1}\pm 2\,,
$$
which is evidently integral.
The claim of lowering the weight \eqref{eq:F}  is equally simple.
\end{proof}

Hence, one may define the following chain of Fibonacci tricycles:
\begin{equation}
\label{eq:reversed}
(\, 0,\, 1, \, 1 \,)
\  \mapsto\ 
(\, -1,\, 2, \, 3 \,)
\  \mapsto\ 
(\, -3,\, 5, \, 8 \,)
\  \mapsto\ 
(\, -8,\, 13, \, 21 \,)
\  \mapsto\ 
\ldots
\end{equation}
%
This thread is {\bf exceptional} in two ways: (1)  unlike the other threads, it is not polynomial;
(2) the proportions of the tricycles 
do not collapse to $(-1,1,0)$ with the increasing $n$, while the others do.
To clarify the second property, consider the thread denoted D.
The ratios of the second and the third disk to the first become in the limit:
$$
\lim_{n\,\to\,\infty} \frac{4n}{4n+4} = 1\,,
\qquad
\lim_{n\,\to\,\infty} \frac{4n}{(2n+1)^2} = 0\,,
$$
which means the greater of the two disks inside is approaching the external disk, 
while the smaller disappears to a point.
The same concerns the other polynomial threads shown in  Figure \ref{fig:paths}.

In contrast,  the triples of the ``Fibonacci thread'' \eqref{eq:reversed}, 
the proportions approach 
the  shape known in Ancient Greece as an ``arbelos''
with a particular split of the diameter, namely the golden cut.
It is shown in Figure~\ref{fig:arbelos}, left, and the Pappus chains on the right side.

\begin{figure}
\centering

\begin{tikzpicture}[scale=50, rotate=116]  
\draw (-2/21,12/21) circle (1/21) ;     
\draw [fill=gold!50] (-3/34, 20/34) circle (1/34);
     \node [scale=1.8] at (-3/34, 20/34) {$\sf \varphi$};
\draw [fill=gold!50] (-6/55, 30/ 55) circle (1/55);
     \node [scale=1.8] at (-6/55, 30/ 55) {$\sf \varphi^2$};
\end{tikzpicture}
\quad
\begin{tikzpicture}[scale=57, rotate=117]  
\draw (-2/21,12/21) circle (1/21) ;     
\draw [fill=gold!50] (-3/34, 20/34) circle (1/34);
     \node [scale=1] at (-3/34, 20/34) {$\sf \varphi$};
\draw [fill=gold!50] (-6/55, 30/ 55) circle (1/55);
     \node [scale=1] at (-6/55, 30/ 55) {$\sf \varphi^2$};
\draw [fill=gold!10] (-5/66, 36/66) circle (1/66);
   \node [scale=1] at (-5/66, 36/66) {$\sf 2 \varphi$};
\draw [fill=gold!10] (-9/70, 40/70) circle (1/70);
    \node [scale=1] at (-9/70, 40/70) {$\sf 2 \varphi$};
         
\draw [fill=gold!10] (-6/103, 58/103) circle (1/103);
   \node [scale=1] at (-6/103, 58/103) [scale=.9] {$\sf 5\varphi\!-\!3$};
\draw [fill=gold!10] (-9/166, 96/166) circle (1/166);
   \node [scale=1] at (-9/166, 96/166) [scale=.5]  {$\sf 10\varphi\!-\!8$};
\draw [fill=gold!10] (-14/255, 150/255) circle (1/255);
\draw [fill=gold!10] (-21/370, 220/370) circle (1/370);
\draw [fill=gold!10] (-30/511, 306/511) circle (1/511);
\draw [fill=gold!10] (-41/678, 408/678) circle (1/678);

\draw [fill=gold!10] (-15/166, 88/166) circle (1/166);
      \node [scale=.9] at (-15/166, 88/166) [scale=.8] {$\sf 5\varphi$};
\draw [fill=gold!10] (-33/334, 176/334) circle (1/334);
     \node [scale=1] at (-33/334, 176/334) [scale=.4] {$\sf 10\varphi$};

\draw [fill=gold!10] (-59/570, 300/570) circle (1/570);
%
\draw [fill=gold!10] (-93/874, 460/874) circle (1/874);

\draw [fill=gold!10] (-30/331, 184/331) circle (1/331);
\draw [fill=gold!10] (-73/774, 432/774) circle (1/774);

\draw [fill=gold!10] (-38/339, 192/339) circle (1/339);
\draw [fill=gold!10] (-85/786, 444/786) circle (1/786);
\end{tikzpicture}
\caption{Golden arbelos and its Pappus chains}
\label{fig:arbelos}
\end{figure}
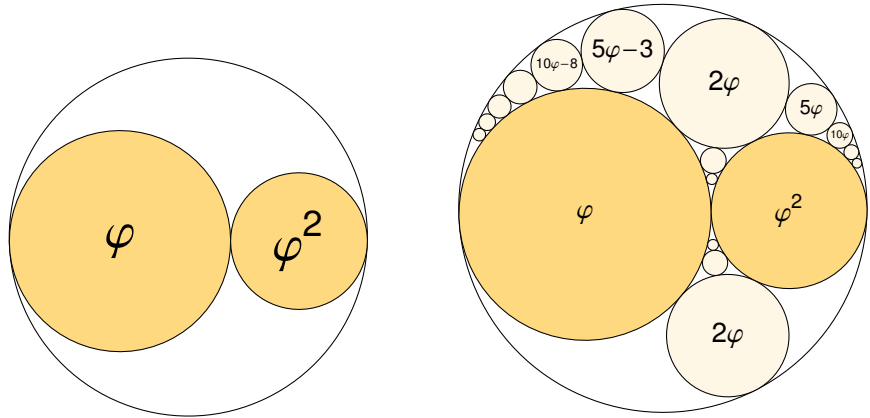

$$
\frac{1}{\varphi} + \frac {1}{\varphi^2} = 1
$$

\noindent
{\bf Remark:}
Every disk in the ``golden arbelos'' has curvature  of the form $m\varphi + n$ for some integers 
$m,n\in\mathbb Z$
as assured that the first four are such.
In particular, the curvatures of the 
 Pappus chain made by disks
passing through the left disk $\varphi$ and going to the  right and  to the left are as follows:
 $$
 \begin{array}{rll}
 (n^2+1)\,\varphi : &\qquad& 
 \varphi^2 = \varphi, \;  2\varphi, \ 5\varphi , \ 10\varphi, \ 17\varphi,\ \ldots  \\[7pt]
 (n^2+1)\varphi -(n^2-1) = \varphi ^2 + \frac{n^2}{\varphi}: &&
 \varphi^2\! =\! \varphi + 1, \;  2\varphi, \ 5\varphi - 3, \ 10\varphi - 8, \ldots 
\end{array}
$$
We suggest extending the term ``Pappus chain'' to any chain of disks 
that is squeezed between two tangent disks.
The Pappus chain
squeezed between the disks $\varphi$ and $\varphi^2$
is:
 $$
2 n^2\varphi +n^2 -1: \qquad 
 -1, \;  2\varphi, \ 8\varphi +3, \ 18\varphi+ 8, \ \ldots 
$$
And clearly, for $n=0$ we get (-1), the greatest enclosing disk.

~

Performing the reversed descending of the Fibonacci thread 
will produce a system of tricycles shown in Figure \ref{fig:chain}.
The points of tangency of the circles are
$$
P_n \  =  \ \frac{F_{n-1}^2, \; 2F_{n}^2} {F_{2n+1}}
$$
They all lie on a single circle
$\frac{-1,\, 1}{1}$
and  have a point limit:
$$
P \ = \ \lim_{n\to\infty} P_n \ = \ \   \left( \frac{5+2\sqrt{5}}{5}  , \;  \frac{5-\sqrt{5}}{5}           \right)
$$
The position of the point is interesting:
Points $(-1,0)$ $(P_x)$, P, $(-1,P_y)$
form a golden rectangle, the one of sides in the golden proportion.
The other rectangle that includes points $(-1,-1), (0,1), P$,  is in the proportion 1:2.
This ``Fibonacci circle'' is interesting on its own and will be presented in a subsequent note.

~

\begin{figure}
\centering

\resizebox{1.05\textwidth}{!}{

\begin{tikzpicture}[scale=7.7]

\clip (-1.05, -.11) rectangle (1.01, 1.05);
\definecolor{gold}{rgb}{1,.70,.0}   

\draw (0,0) circle (1);
\draw [fill=blue!10] (-1/2,0) circle (1/2);
\draw [fill=blue!10] (0,2/3) circle (1/3);

\draw [fill=blue!20] (0/5,4/5) circle (1/5);
\draw[fill=blue!20]  (-1/8,4/8) circle (1/8);

\draw [fill=blue!30] (-2/13,6/13) circle (1/13);
\draw [fill=blue!30]  (-2/21,12/21) circle (1/21);

\draw [fill=blue!40] (-3/34,  20/34) circle (1/34);
\draw [fill=blue!40]  (-6/55,  30/55) circle (1/55);

\draw [fill=blue!77] (-10/89,  48/89) circle (1/90);
\draw [fill=blue!77] (-15/144,  80/144) circle (1/150);


\node [scale= 3]  at (0.6,0)  {$\sf 1$};
\node [scale= 3]  at (-0.5,0)  {$\sf 2$};
\node [scale= 2.5]  at (0.38,  2/3-.12)  {$\sf 3$};
\node [scale= 3]  at (0.045,  4/5)  {$\sf 5$};
\node [scale= 1.8]  at (-1/8+.16,4/8-.07)  {$\sf 8$};
\node [scale= .8 ]  at (-1/8+.1,4/8+.03)  {$\sf 21$};
\node [scale= 2]  at (-2/13, 6/13)  {$\sf 13$};
\node [scale= 1] at (-3/34,  20/34) {$\sf 34$};

\draw [red,thick] (-1,1) circle (1);
        
\end{tikzpicture}
}

\caption{The chain of the Fibonacci thread, realized in the plane.}
\label{fig:chain}
\end{figure}
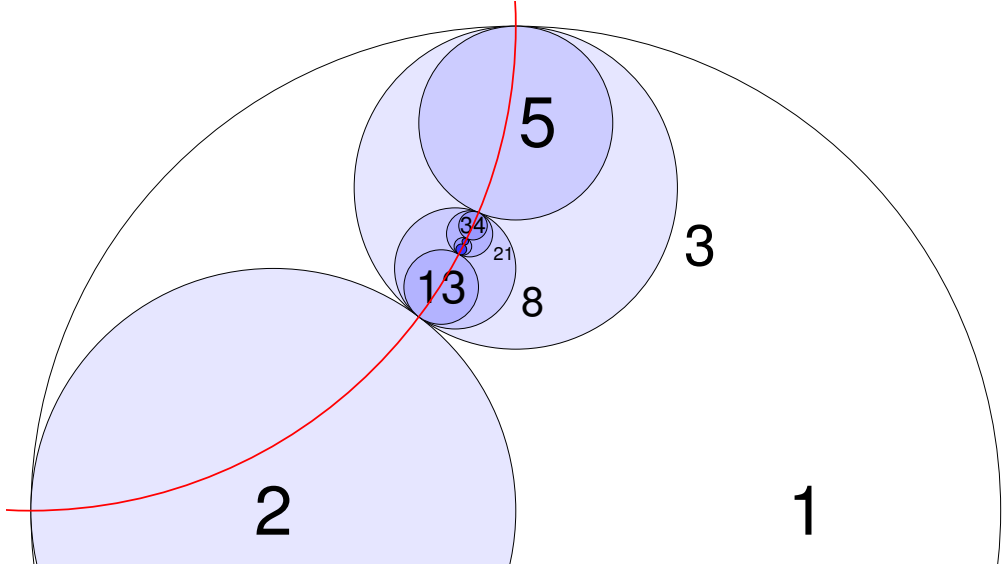

Whether there exist other exceptional threads is an interesting question. 

\newpage
\section{The group-theoretic aspect}

In this section we rewrite the story with the group-theoretic context clarified.

First, one needs to distinguish between the tricycles 
as the actual geometric triads of disks located in $\mathbb R^2$, 
and the triples of numbers representing their curvatures that can be treated algebraically.
We start with the latter.  Equip $\mathbb R^3$ with inner product defined by matrix 
\begin{equation}
\label{eq:smallG}
G =\frac{1}{2}\, \begin{bmatrix}  0&1&1\\
                               1 & 0 &1\\
                               1 & 1& 0
      \end{bmatrix}
\end{equation}      
In explicit notation, the quadratic form is:
$$
Q(\mathbf v)\equiv |\mathbf v|^2 = xy+yz+zx  \qquad \hbox{for} \ \  \mathbf v=[x,y,z]^T
$$ 
The space $(\mathbb R^3, G)$ is isomorphic to 
a two-dimensional Minkowski space $\mathbb R^{1,2}$
of signature $(+,-,-)$,
as can be demonstrated with  $4Q= (x+2y+z)^2 - (x-z)^2-(2y)^2$.
Define the subset of triples of integers: 
$$
\mathcal T = 
\Big\{\, \mathbf v\in \mathbb Z^3 \;\mathbf:\; 
                        |\mathbf v|\in \mathbb N,\, w(\mathbf v)>0,\, \gcd(v) = 1\;\Big\}
                         \quad \subset \mathbb Z^3
$$
where $\mathbb N$ includes zero.
The three restrictions are repeated here explicitly
$$
\begin{array}{cll}
(i) && w(\mathbf v) = a+b+c \, > \, 0\\ 
(ii) && |\mathbf v| \ \equiv \ \sqrt{ab+bc+ca\;}   \quad   \in \ \mathbb N\!=\!\{0,1,2...\}\\
(iii) && \gcd(\mathbf v) \equiv \gcd(a,b,c) = 1
\end{array}
$$ 
These conditions imply that such triples can be realized as curvatures of disks in tricycles.  
In particular, that at most one number can be negative 
and its absolute value smaller than any of the two remaining,
and that tricycle is proper, i.e., it may be completed to an integral Descartes quadruple  
due to ({\it ii}).
One may think of them as tricycles with unassigned position or orientation 
in the Euclidean plane $\mathbb R^2$.

There are two groups acting in this space:
$$
\begin{array}{cl}
\APt &= \hbox{the group generated by Descartes moves}  \\
\SIt &= \hbox{the group generated by self-inversions} 
\end{array}
$$
The group generated by the union of generators of both groups will be called 
the {\bf Descartes tricycle group} and this fact will be denoted with the symbol ``$\dot\cup$'':
$$
           \Dest   \  = \  \APt \ \dot\cup \ \SIt
$$
It acts transitively on the set $\mathcal T$,
which is the main point of Proposition~\ref{thm:main}.
The notation and terminology in a wider context is addressed in the Appendix. 
A short description of both groups  follows.

~\\
1. The {\bf kaleidoscope group} $\SIt$ generated by {\bf self-inversions}
represented by linear transformations given by the following matrices: 
\begin{equation}
\label{eq:M}
M_1 = \begin{bmatrix}  -1&0&0\\
                                 \; 2 & 1 &0\\
                                 \; 2 & 0&1
      \end{bmatrix}\,,
      \qquad
M_2 = \begin{bmatrix}  1&\; 2&0\\
                                 0 & -1 &0\\
                                 0 & \; 2&1
      \end{bmatrix}\,,
\qquad
M_3 = \begin{bmatrix}  1&0&\; 2\\
                               0 & 1 &\; 2\\
                               0 & 0&-1
      \end{bmatrix}\,.
\end{equation}
The generators satisfy $M_1^2=M_2^2=M_3^2=I$  (identity matrix)
and preserve the inner product~\eqref{eq:smallG}:
$$
M_i^TGM_i=G \qquad\forall i=1,2,3\,.
$$
thus, they generate a discrete subgroup of the Lorentz group  
$$
 \SIt \ = \  \gen\{M_1,M_2,M_3\} \  \subset \  {\rm O}(G,\mathbb Z) \ \cong \ {\rm O}_{(1,2)}(\mathbb Z)
$$
2.  The {\bf nonlinear Apollonian group}  generated by {\bf tricycle Descartes moves} $\APt$  
represented by 6 generators:
$$
 \APt \ = \  \gen\{\,N_{1\pm},\,N_{2\pm},\, N_{3\pm}\,\} 
$$
where 
\begin{equation}
\label{eq:NNN}
\begin{array}{rl}
N_{1\pm}(a,b,c) &= (w(\mathbf v) \pm 2|\mathbf v|,\,b,\,c),\\
N_{2\pm}(a,b,c) &= (a,\,w(\mathbf v) \pm 2|\mathbf v|,\,c), \\
N_{3\pm}(a,b,c) &= (a,b,w(\mathbf v) \pm2|\mathbf v|)
\end{array}
\end{equation}
The group preserves the set $\mathcal T$, and transfers tricycles in a transitive way 
to other tricycles in the same Apollonian packing, hence the adjective ``Apollonian'' in the name.
The star is to indicate that it is nonlinear 
and is not equivalent to the well-known Apollonian group of \cite{Lagarias+}.

~\\
{\bf Graph structure.}
The Cayley graph of the groups are presented in Figure~\ref{fig:valent}.
The groups impose a similar structure on the set of tricycles $\mathcal T$,
which may be viewed as a graph with tricycles as vertices,
and the edges defined by the group action:
\\[-7pt]

Type 1:  There exists an  edge $(a,b)$ if \ $b=M_ia$ \  for some $M_i$.

Type 2: There exists  an edge $(a,b)$ if \  $b=N_ia$ \ for some $N_i$ .
\\[-7pt]
\\
The type 1 subgraph is three-valent and forms an infinite tree shown in Figure~\ref{fig:valent}, left.
Such a graph is known as $z=3$ Bethe lattice.
Type 2 subgraph is 6-valent and is not cycle-free.
(But it may be viewed as a cycle-free infinite  ensemble of tetrahedrons joined at their vertices.)
Every regular point of $\mathcal T$ is thus a 9-valent vertex,
with six edges going along the fiber, and three across the fiber bundle.
The edges may be oriented by the weight function:
$$
a\to b \qquad \hbox{if} \qquad w(b)<w(a)\,.
$$

\begin{figure}[H]
\centering
\begin{tikzpicture}[scale=.3] 
\def\H{.866}
\def\V{.5}
\draw  [fill=black] (0,0) circle (.5);   
\draw [] (0,0) -- (0, 4);  
\draw [] (0,0) -- (4*.866,-4*.5);  
\draw [] (0,0) -- (-4*.866,-4*.5);  
\draw  [fill=black] (0,4) circle (.4);   
\draw  [fill=black] (4*\H,-4*\V) circle (.4);  
\draw  [fill=black] (-4*\H,-4*\V) circle (.4); 

\draw [] (4*\H,-4*\V) -- (7*\H,-\V);  
\draw [] (4*\H,-4*\V) -- (4*\H,-4*\V-3);  
\draw  [fill=black] (7*\H,-\V) circle (.3);   
\draw  [fill=black] (4*\H,-4*\V-3) circle (.3);   
%
\draw [] (-4*\H,-4*\V) -- (-7*\H,-\V);  
\draw [] (-4*\H,-4*\V) -- (-4*\H,-4*\V-3);  
\draw  [fill=black] (-7*\H,-\V) circle (.3);   
\draw  [fill=black] (-4*\H,-4*\V-3) circle (.3);   
%
\draw [] (0,4) -- (3*\H,4+3*\V);  
\draw [] (0,4) -- (-3*\H,4+3*\V);  
\draw  [fill=black] (3*\H,4+3*\V) circle (.3);   
\draw  [fill=black] (-3*\H,4+3*\V) circle (.3);   
%

\draw [] (7*\H,-\V) -- (7*\H,-\V+2);  
\draw [] (7*\H,-\V) -- (9*\H,-3*\V);  
\draw  [fill=black] (7*\H,-\V+2) circle (.3);   
\draw  [fill=black] (9*\H,-3*\V) circle (.3);   

\draw [] (-7*\H,-\V) -- (-7*\H,-\V+2);  
\draw [] (-7*\H,-\V) -- (-9*\H,-3*\V);  
\draw  [fill=black] (-7*\H,-\V+2) circle (.3);   
\draw  [fill=black] (-9*\H,-3*\V) circle (.3);   

\draw [] (4*\H,-4*\V-3) -- (6*\H,-6*\V-3);  
\draw [] (4*\H,-4*\V-3) -- (2*\H,-6*\V-3);  
\draw  [fill=black] (6*\H,-6*\V-3) circle (.3);   
\draw  [fill=black] (2*\H,-6*\V-3) circle (.3);   
\draw [] (-4*\H,-4*\V-3) -- (-6*\H,-6*\V-3);  
\draw [] (-4*\H,-4*\V-3) -- (-2*\H,-6*\V-3);  
\draw  [fill=black] (-6*\H,-6*\V-3) circle (.3);   
\draw  [fill=black] (-2*\H,-6*\V-3) circle (.3);   

\draw [] (3*\H,4+3*\V) -- (5*\H, 4+\V);  
\draw [] (3*\H,4+3*\V) -- (3*\H,6+3*\V);  
\draw  [fill=black] (5*\H, 4+\V) circle (.3);   
\draw  [fill=black] (3*\H,6+3*\V) circle (.3);   
%
\draw [] (-3*\H,4+3*\V) -- (-5*\H, 4+\V);  
\draw [] (-3*\H,4+3*\V) -- (-3*\H,6+3*\V);  
\draw  [fill=black] (-5*\H, 4+\V) circle (.3);   
\draw  [fill=black] (-3*\H,6+3*\V) circle (.3);   
\end{tikzpicture}
\qquad\quad
\begin{tikzpicture}[scale=.2, rotate=0]
\def\s{.5}
\foreach \a  in {
4}
{\draw [-, thick]  (-\a, -\a) -- (\a, -\a) -- (\a,\a) -- (-\a,\a) -- cycle;
\draw [-, thick]  (-\a,-\a) -- (\a, \a);
\draw [-] (\a, -\a) -- (-\a, \a);
\draw [fill=black] (\a,\a) circle (\s) (-\a,\a) circle (\s)
          (\a,-\a) circle (\s)  (-\a,-\a) circle (\s);
}; 
\foreach \m/\n in {  1/1,  -1/1,  1/-1,  -1/-1 }
\foreach \a/\b/\c\s in { 4/4/4/.4,  8/8/2/.3, 2/8/2/.3, 8/2/2/.3,
10/10/1/.2, 7/10/1/.2, 10/7/1/.2, 4/10/1/.2, 10/4/1/.2, 1/10/1/.2, 10/1/1/.2, 1/7/1/.2, 7/1/1/.2 }
{
\draw [-, thick]  (\m*\a, \n*\b) rectangle (  \m*\a+\m*\c,  \n*\b+\n*\c );
    \draw [-, thick]  (\m*\a+\m*\c,\n*\b) -- (\m*\a, \n*\b+\n*\c);
    \draw [-, thick]  (\m*\a,\n*\b) -- (\m*\a+\m*\c, \n*\b+\n*\c);
\draw [fill=black] 
             (\a*\m,\b*\n) circle (\s)  (\a*\m+\c*\m,\b*\n) circle (\s)  
             (\a*\m,\b*\n+\c*\n) circle (\s)  (\a*\m+\c*\m,\b*\n+\c*\n) circle (\s);
};
\end{tikzpicture}

\caption{Discrete structure of the groups.  Left: generated by the three self-inversions;
Right: generated by the six Descartes moves. 
It also reflects the local structures of $\mathcal T$ as a graph. Left:along fibers, right: across fibers.}
\label{fig:valent}
\end{figure}
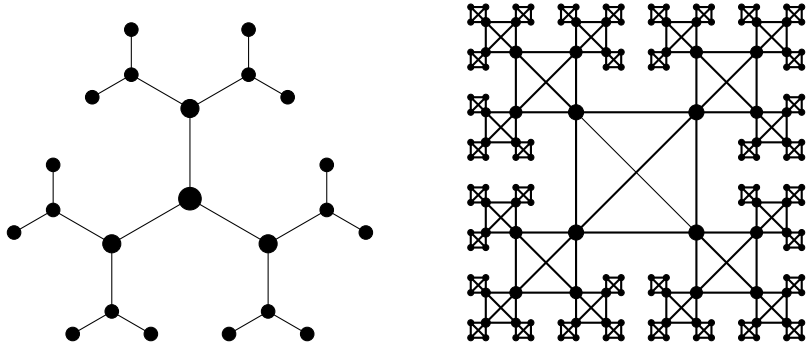

\noindent
{\bf The fiber geometry of the space of tricycles.}
The action of the group of Descartes moves turns  $\mathcal T$ into a fiber bundle 
with the natural projection
\begin{equation}
\label{eq:piA}
\pi_{\mathcal A}: \ \mathcal T \ \to \ \mathcal T/\SIt \ \equiv \ \mathcal A
\end{equation}
where the factor space 
$$
\mathcal A = \mathcal T \big/ \APt\ \cong \ \{\hbox{\rm all integral Apollonian packings}\} 
$$
can be identified with the set of the integral irreducible Apollonian disk packings,
and the fibers are the group orbits.
Each fiber contains all triples in the particular Apollonian packing.
Thus, for a given tricycle $t$, the set
$$
\pi^{-1}\circ\pi \,(\mathbf t) \ \cong \  \APt\cdot \mathbf t
$$
consists of all triples in the Apollonian packing that contains $t$
(with no particular location in $\mathbb R^2$, of course). 
 
Every fiber has a weight-minimal element coinciding with the principal tricycle of the 
corresponding Apollonian packing.
Note that starting with an arbitrary vertex in the fiber and moving down the weight 
is what defines the ``depth function'' discussed in \cite{jk-depth,jk-universal}.

\begin{figure}[H]
\centering
\hspace*{-.2in}
\begin{tikzpicture}[scale=.7]  

\draw [thick] (0,1.2) rectangle (7.7,5.3);
\draw [very thick]  (0,0) -- (7.7,0); 
\draw node  [scale=1.5]  at (8.2,4.5)  {$\mathcal T$};
\draw node  [scale=1.5]  at (8.2,0)  {$\mathcal A$};

\draw [black,->, line width=.7mm] (4, 1) -- (4, .15);
\draw node  [scale=1.1]  at (4.5,.55)  {$\pi_{\mathcal A}$};

\draw [blue,very thick]  (2,1.2) -- (2,5.3); 
\draw node  [scale=1.1]  at (2.56,4.2)  {$\mathcal T_{\alpha}$};

\draw [blue,very thick,dotted]  (2,1) -- (2,.2); 

\draw [fill=black] (2,3) circle (.1); 
\draw node  [scale=1.1]  at (2.5,3)  {$\mathbf t$};
\draw [fill=black] (2,0) circle (.1); 
\draw node  [scale=1.1]  at (2,-.5)  {$\alpha$};

\draw [green,<->, line width=1.2mm] (5.1, 4.2) -- (6.9, 4.2);
\draw node  [scale=1]  at (6.1,4.8)  {$\SIt$};

\draw [green,<->, line width=1.2mm] (6.2, 1.7) -- (6.2, 3.5);
\draw node  [scale=1,rotate=90]  at (5.5,2.6)  {$\APt$};

\draw [fill=black] (0,1.2) circle (.1); 
\draw node  [scale=1.1]  at (0,.7)  {$\mathbf t_0$};

\draw [fill=black] (0,0) circle (.1); 
\draw node  [scale=1.1]  at (0,-.5)  {$\alpha_0$};

\end{tikzpicture}
\qquad
\begin{tikzpicture}[scale=.7]  
\draw [fill=gray!47] (0,1.2) rectangle (7,2.8);
\draw [thick] (0,1.2) rectangle (7,5.3);
\draw [very thick]  (0,0) -- (7,0); 
\draw node  [scale=1.5]  at (3,4.7)  {$\mathcal T$};
\draw [black,->, line width=.7mm] (4, .9) -- (4, .1);
\draw node  [scale=1.1]  at (4.5,.5)  {$\pi_{\mathcal A}$};

\draw [red, thick] (6,4) -- (6,2.8) -- (5,3.5) -- (5,2.8) -- (4,1.8) 
-- (3,1.6) -- (1.5,4) -- (1.5,2.8) -- (0,3.8) -- (0,1.2); 

\draw [-{>[scale=1.1]}, red, thick] (6,4) -- (6,3.5);
\draw [-{>[scale=1.1]}, red, thick] (0,3.8) -- (0,2);

\draw [fill=black] (6,4) circle (.1); 
\draw node  [scale=1.1]  at (6.5,4)  {$\mathbf t$};

\draw [fill=black] (6,2.8) circle (.1); 
\draw [fill=black] (5,3.5) circle (.1); 
\draw [fill=black] (5,2.8) circle (.1); 
\draw [fill=black] (4,1.8) circle (.1); 
\draw [fill=black] (3,1.6) circle (.1); 
\draw [fill=black] (1.5,4) circle (.1); 
\draw [fill=black] (1.5,2.8) circle (.1); 
\draw [fill=black] (0,3.8) circle (.1); 
\draw [fill=black] (0,1.2) circle (.1); 

\draw [fill=black] (6,0) circle (.1); 
\draw node  [scale=1.1]  at (6,-.5)  {$\alpha$};

\draw [fill=black] (0,1.2) circle (.1); 
\draw node  [scale=1.1]  at (0,.7)  {$\mathbf t_0$};

\draw [fill=black] (0,0) circle (.1); 
\draw node  [scale=1.1]  at (0,-.5)  {$\alpha_0$};

\draw [red, thick] (0,3.8) -- (0,1.2); 

\end{tikzpicture}
\caption{Left: The space of triples as a fiber bundle. Apollonian belt is denoted $\alpha_0$.
Right: A weight-descending path.}
\label{fig:fiberbundle}
\end{figure}
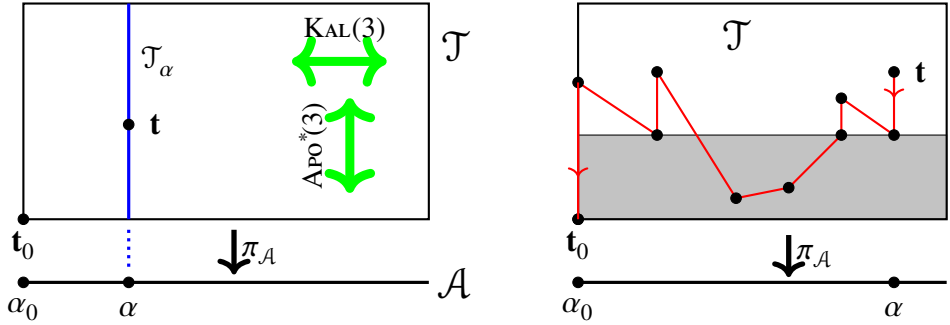

Figure \ref{fig:fiberbundle}, left, shows the fiber bundle structure \eqref{eq:piA}.
The groups are acting ``vertically'' and ``horizontally''. That is, 
elements of $\SIt$ permute the elements of the fibers
but do not permute fibers 
$$
\pi_A\circ \SIt = \pi_{\mathcal A}
$$
while the  action of $\SIt$ does not preserve the fiber structure,
i.e., in general, non-identity elements $g\in\SIt(2)$ skater element of a fiber
to different fibers, i.e. in general
$$
\pi_{\mathcal A}(t) \ \not= \  \pi_{\mathcal A}(g(t))
$$
with exception when $\mathbf t\in \pi^{-1}_{\mathcal A}(\mathbf t_0)$ contains a 0.
Group $\APt$ acts transitively on on each fiber, and 
group $\Dest = \APt\,\dot\cup\, \SIt$ acts transitively on $\mathcal T$.

~

\noindent
{\bf Geometric meaning of the inner product.}  
The inner product \eqref{eq:smallG}  has a clear geometric meaning.
As is well known, for a given three disks of curvatures
$a,b,c$, the curvature of the circle $K$ going through the points of tangency
is $\sqrt{ab+bc+ca\,}$, which coincides with $|\mathbf v|$, where $\mathbf v = [a,b,c]^T$.
The preservation of  the inner product by the self-inversions is beautifully illustrated by the fact that 
the circle $K$ is preserved by these transformations.
Figure \ref{fig:chainBig}, left, shows tricycle of disks with curvatures 11, 14, 15,  
and the inversions of the last two through ``11'', assuming curvatures 36 and 37.
The limit of the iteration of self-inversions of the original disks is the circle $K$;
see the right side of the figure.

\begin{figure}[h]
\centering
\begin{tikzpicture}[scale=14,rotate=150]
\%clip (-1.05, -.11) rectangle (1.01, 1.05);

\definecolor{gold}{rgb}{1,.70,.0}   

\def\disk#1#2#3{%
\draw [black,fill=blue!20, opacity=.5] (#1/#3,#2/#3) circle (1/#3);%
\draw [black] (#1/#3,#2/#3) circle (1/#3);%
\draw node at (#1/#3,#2/#3) [scale=1.2] {$\sf #3$};
}

\def\diskt#1#2#3{%
\draw [black,fill=blue!30, opacity=.6,thin] (#1/#3,#2/#3) circle (1/#3);%
\draw [black,thin] (#1/#3,#2/#3) circle (1/#3);%
\draw node at (#1/#3,#2/#3) {$\sf #3$};
}

\disk{6}{8}{11};
\disk{7}{8}{14};
\disk{6}{10}{15};

\diskt{19}{24}{36};
\diskt{18}{26}{37};
\draw [red,thick] (11/23,15/23) circle (1/23);
        
\end{tikzpicture}
\qquad
%
\begin{tikzpicture}[scale=14,rotate=150]
\%clip (-1.05, -.11) rectangle (1.01, 1.05);

\definecolor{gold}{rgb}{1,.70,.0}   

\def\disk#1#2#3{%
\draw [black,fill=blue!20, opacity=.5] (#1/#3,#2/#3) circle (1/#3);%
\draw [black] (#1/#3,#2/#3) circle (1/#3);%
}

\def\diskt#1#2#3{%
\draw [black,fill=blue!30, opacity=.6,thin] (#1/#3,#2/#3) circle (1/#3);%
\draw [black,thin] (#1/#3,#2/#3) circle (1/#3);%
}

\disk{6}{8}{11};
\disk{7}{8}{14};
\disk{6}{10}{15};
\disk{19}{24}{36};
\disk{18}{26}{37};
\disk{20}{24}{39};
\disk{20}{26}{43};
\disk{18}{28}{41};
\disk{19}{28}{44};
\disk{	32	}{	40	}{	61	}
\disk{	56	}{	74	}{	109	}
\disk{	30	}{	44	}{	63	}
\disk{	55	}{	76	}{	110	}
\disk{	33	}{	40	}{	64	}
\disk{	60	}{	74	}{	121	}
\disk{	60	}{	76	}{	125	}
\disk{	33	}{	44	}{	72	}
\disk{	32	}{	46	}{	73	}
\disk{	56	}{	84	}{	129	}
\disk{	30	}{	46	}{	67	}
\disk{	55	}{	84	}{	126	}
\diskt{	45	}{	56	}{	86	}
\diskt{	120	}{	154	}{	231	}
\diskt{	93	}{	124	}{	182	}
\diskt{	144	}{	188	}{	279	}

\diskt{	42	}{	62	}{	89	}
\diskt{	115	}{	164	}{	236	}
\diskt{	92	}{	126	}{	183	}
\diskt{	140	}{	196	}{	283	}
\diskt{	46	}{	56	}{	89	}
\diskt{	126	}{	154	}{	249	}
\diskt{	100	}{	124	}{	203	}
\diskt{	153	}{	188	}{	306	}
\diskt{	100	}{	126	}{	207	}
\diskt{	153	}{	196	}{	322	}
\diskt{	46	}{	62	}{	101	}
\diskt{	126	}{	164	}{	269	}
3e
\diskt{	45	}{	64	}{	102	}
\diskt{	120	}{	176	}{	275	}
\diskt{	93	}{	140	}{	214	}
\diskt{	144	}{	214	}{	331	}
\diskt{	42	}{	64	}{	93	}
\diskt{	115	}{	176	}{	260	}
\diskt{	92	}{	140	}{	211	}
\diskt{	140	}{	214	}{	319	}

\draw [yellow,thick] (11/23,15/23) circle (1/23);
        \end{tikzpicture}
\caption{The limit set of the action of the tricycle group.}
\label{fig:chainBig}
\end{figure}
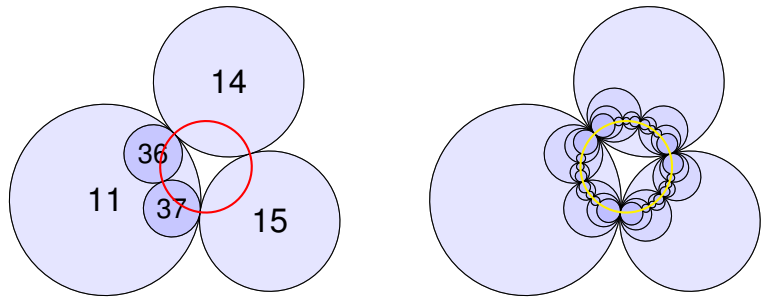

\noindent
{\bf Geometrization of the triples.}
Now to produce the hyper-integral Apollonian packing in the Euclidean plane given
an integral packing devoid of any information o locations, we need another fiber bundle.
Let $\mathcal T_{\sf geo}$ represent the set of (integral proper) tricycles located in $\mathbb R^2$.
Among others, it contains the tricycle denoted $\bar{\mathbf t}_0$ visualized in Figure~\ref{fig:bottom}.
There is a group of Euclidean motions ${\rm E}(2)$
(group of isometries, generated by rotations, translations, and reflections) acting on this set. 
The group action turns the set into a fiber bundle over $\mathcal T\cong \mathcal T_G/{\rm E}(2)$,
with projection
$$
\pi_{\mathcal T} : \mathcal T_{\sf geo} \ \to \ \mathcal T
$$
There is a natural lift of the action of the groups 1 and 2 to the bundle 
defined as follows.
Elements of $\mathcal T_{\sf geo}$ may be identified with matrices 
$$
t = \begin{bmatrix}
\dot x_A & \dot y_A & A\\
\dot x_B & \dot y_B & B\\
\dot x_C & \dot y_C & C\\
\end{bmatrix}
$$
In the case of self-inversions, 
the generators are defined by exactly the same 
matrices as in \eqref{eq:M}.
In this context they will be denoted with over-bars:  ${\bar M}_1$, ${\bar M}_2$, and ${\bar M}_3$, respectively.
In the case of the action of the self-inversion group $\SIt(2)$, 
the action is defined as 
$$
{\bar N}_1: \ 
\begin{bmatrix}
\dot x_A & \dot y_A & A\\
\dot x_B & \dot y_B & B\\
\dot x_C & \dot y_C & C\\
\end{bmatrix}
\ 
\to
\ 
\begin{bmatrix}
f(\dot x_A,\dot x_B,\dot x_C) & f(\dot y_A,\dot y_B, \dot y_C) & g(A,B,C)\\
\dot x_B & \dot y_B & B\\
\dot x_C & \dot y_C & C\\
\end{bmatrix}
$$
where
$$
g_{\pm} (\mathbf v) = w(\mathbf v) \pm 2\sqrt{Q(\mathbf v)\,}
\qquad\hbox{and}\qquad
f_\pm(\mathbf v) = w(\mathbf v) \pm2\sqrt{Q(\mathbf v)+1\,}\\
$$
or, explicitly, 
$$
\begin{array}{rl}
f_\pm(x,y,x) &= x+y+z\pm2\sqrt{xy+yz+zx+1\,}\\
g_{\pm} (x,y,x) &= x+y+z\pm 2\sqrt{xy+yz+zx\,}
\end{array}
$$
The other generators are defined analogously.

~

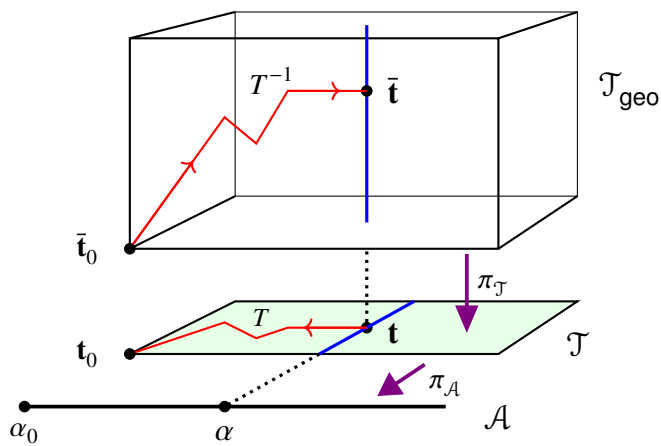
\begin{figure}[H]
\centering
\hspace*{-.2in}
\begin{tikzpicture}[scale=.7]  
\draw [line width=.49mm]  (0,0) -- (8,0); 
\draw [thick, fill=green!10] (2,1) -- (9,1) -- (10.5,2) -- (4,2) -- (2,1);
\draw [] (4,4) rectangle (10.5,7.5);
\draw [thick] (2,3) rectangle (9,7);
\draw [thick] (2,3) -- (4,4);  
\draw [thick] (2,7) -- (4,7.5);  
\draw [thick] (9,3) -- (10.5,4);  
\draw [thick] (9,7) -- (10.5,7.5);  
        \draw node  [scale=1.2]  at (9,-.2)  {$\mathcal A$};
        \draw node  [scale=1.2]  at (10.5,1.2)  {$\mathcal T$};
        \draw node  [scale=1.2]  at (11.5,6)  {$\mathcal T_{\sf geo}$};
\draw [violet,-triangle 45, line width=.5mm] (8.4, 2.9) -- (8.4, 1.4);
      \draw node at (8.9, 2.3) {$\pi_{\mathcal T}$};
\draw [violet,-triangle 60, line width=.5mm] (7.6, .8) -- (6.7, .2);
      \draw node at (8, .4) {$\pi_{\mathcal A}$};
\draw [very thick, blue] (6.5,7.25) -- (6.5,3.5);
\draw [fill=black] (6.5, 6) circle (.1); 
      \draw node [scale=1.1] at (7 , 6) {$\bar {\mathbf t}$}; 
\draw [very thick, dotted] (6.5,2.95) -- (6.5,1.5); 
\draw [fill=black] (6.5,1.5) circle (.1);
       \draw node [scale=1.1] at (7 , 1.4) {$\mathbf t$}; 
\draw [very thick, blue] (7.4,2) -- (5.6,1);
\draw [very thick, dotted] (5.6,1) -- (3.8, 0); 
\draw [fill=black] (3.8, 0) circle (.1); 
        \draw node [scale=1.1] at (3.8, -.5) {$\alpha$}; 
\draw [red, thick]  (2,1) --(3.8,1.6) --   (4.4,1.3) --  (5,1.5) -- (6.5,1.5);  
\draw [<-, red, thick]  (5.3,1.5) -- (6.5,1.5);
\draw [red, thick] (2,3) --(3.8,5.5) --  (4.4,5) -- (5, 6)   -- (6.5,6);
\draw [->, red ,thick]  (2,3) -- (3.8-.6,5.5-2.5/3);
\draw [->, red , thick]  (5, 6) -- (6,6);
        \draw node [scale=.9] at (4.5, 1.7) {$T$}; 
        \draw node [scale=1] at (4.7, 6.2) {$T^{-1}$}; 
\draw [fill=black] (2,3) circle (.1); \draw node at (1.2,3) {$\bar{\mathbf t}_0$};
\draw [fill=black] (2,1) circle (.1);  \draw node at (1.2,1) {$\mathbf t_0$};
\draw [fill=black] (0,0) circle (.1);  \draw node at (0,-.5) {$\alpha_0$};
\end{tikzpicture}
\caption{The space of geometric tricycles as a fiber bundle.  Point $t$ represents the starting triple, 
$t_0$ --- the base triple $(0,0,1)$,
$\bar t_0$ -- the base tricycle in the hyper-integral location,
$T$ denotes the transformation that brings $t$ to $t_0$,
$T^{-1}$ its inverse applied to $\bar t_0$ and bringing it to the hyper-integral location $\bar{\mathbf t}$
}
\label{fig:bigbundle}
\end{figure}

~\\
{\bf Descending and ascending path.}
On the right side of Figure~\ref{fig:fiberbundle}, 
the darker region symbolizes the unbounded number triples,
i.e. containing negative or zero entries.
The polygonal path that starts at arbitrary tricycle $t$ 
(that can be selected from some Apollonian packing $\alpha$)
illustrates the process described in Section \ref{sec:main}.
At every step, if the tricycle is regular (contain no negative curvature disk),
apply Descartes move
(this moves the triple along the fiber $\pi^{-1}(\alpha)$).
Repeat until the first triple with a negative entry is hit.
In such a case apply the self-inversion.
Repeat the steps until the base triple $\mathbf t_0=(0,0,1)$ is reached.   
The process creates an element $T\in \Dest$ such that 
$T(\mathbf t) = \mathbf t_0$.
Apply its inverse to the geometrized version of the base tricycle to get 
$\bar{\mathbf t}= T^{-1} (\bar{\mathbf t}_0)$,
which is a tricycle with assigned hyperintegral location in the plane,
which may be completed to the corresponding Apollonian packing.

\newpage

\section{Proof done with quadruples and Descartes group}

The hyper-integrality proof of the previous sections was based on tricycles, 
the ``primitive atoms'' of Apollonian packings.
But the story may easily be reformulated in terms of Descartes configurations, 
the ``molecules'' of the Apollonian packings,
which are more familiar in the context of the Apollonian packings.
Below we provide such a formulation in a succinct way.
As an interesting results, a symmetry group with linear representation, 
containing the well-known Apollonian group, emerges.

~\\
$\bullet$ 
{\bf Algebra:  integral quadruples.} Let $\mathcal D\subset \mathbb Z^4$ denote the subset of integer quadruplets, 
written usually as  $\mathbf v=(a,b,c,d)$ or alternatively as the column vector $\mathbf v=[a \ b \ c \ d ] ^T$,
defined by 
\begin{equation}
\label{eq:D}
\mathcal D \ = \ \{\mathbf v\in\mathbb Z^4  : \ \gcd(\mathbf v) = 1, \ Q(\mathbf v) = 0, \ w(\mathbf v)>1 \ \}
\end{equation}
where
$$
\begin{array}{cll}
(i) && \gcd(\mathbf v) =\gcd(A,B,C,D) \\
(ii) && Q(v)=
(A\!+\!B\!+\!C\!+\!D)^2\!-\!2(A^2\!+\!B^2\!+\!C^2\!+\!D^2\;   \\
(iii) && w(\mathbf v) =  a+b+c+d \qquad \hbox{denotes the weight}
\end{array}
$$ 
The quadratic form $Q$ of (ii) corresponds to Gram matrix
\begin{equation}
\label{eq:Q4}
\small
\setlength\arraycolsep{3pt}
G = \begin{bmatrix*}[r]  -1&1&1&1\;\\
                                 \;1& -1 &1&1\;\\
                                 \;1  &1&-1&1\; \\
                                 \;1 & 1 & 1 &-1\;
      \end{bmatrix*}\,,
\end{equation}
which defines an inner product in $\mathbb R^4$.
Notation $Q(\mathbf v) = \|\mathbf v\|^2 = \mathbf v^TG\mathbf v$ may also be used. 
The conditions ({\it i})--({\it ii}) are to capture only primitive quadruplets 
that can be realized as curvatures of disks in contact. 
In particular, the second condition is equivalent to the Descartes formula.
The last condition excludes the ``non-geometric'' examples like $(1,-2,-2,-3)$,
which satisfy Descartes formula but are not ``geometric''. 
The lowest weight is assumed by $\mathbf t_0 = (0,0,1,1)$ up to the permutation of the entries,
the ``base quadruple.''
\\
\\
{\bf Remark:}  
The $G$ turns $\mathbb R^4$ into a Minkowski space 
of signature $(+,-,-,-)$, which is dual to that discussed in \cite{jk-Descartes}.
The condition $Q=0$ makes the vectors of $\mathcal D$ isotropic.

~

\noindent
$\bullet$ {\bf Geometry: Descartes configurations.}  By  $\mathcal D_{\sf geo}$, 
we denote geometric instances of the quadruples of $\mathcal D$,
i.e., the set of concrete Descartes configurations in the plane with explicit position.
The elements of $\mathcal D$ may be described by matrices
\begin{equation}
\label{eq:piD}
\small
\mathbf q = \begin{bmatrix}
\dot x_A & \dot y_A & A&A^c\\
\dot x_B & \dot y_B & B&B^c\\
\dot x_C & \dot y_C & C&C^c\\
\dot x_D & \dot y_D & D&D^c
\end{bmatrix}
\qquad
\begin{tikzpicture}[scale=3.7, anchor=base,baseline=-7.7ex, rotate=150]  
\draw [fill=yellow, opacity=.3, thick] (6/11,8/11) circle (1/11);
\draw [fill=yellow, opacity=.3, thick] (7/14,8/14) circle (1/14);
\draw [fill=yellow, opacity=.3, thick] (6/15,10/15) circle (1/15);
\draw [fill=yellow, opacity=.3, thick] (41/86,56/86) circle (1/86);
\end{tikzpicture}
\end{equation}
where $\dot x$, $\dot y$ are the reduced coordinates,
and the third and the last column contain curvatures and co-curvatures, respectively.  
In a sense, it is vector $\mathbf v$ of \eqref{eq:D} ``extended''  to a matrix.
We have a natural fibration
$$
\pi_{\mathcal D} :  \mathcal D_{\sf geo} \ \to \ \mathcal D:
\qquad
\tiny\setlength\arraycolsep{2pt}
\begin{bmatrix}
\dot x_A & \dot y_A & A&A^c\\
\dot x_B & \dot y_B & B&B^c\\
\dot x_C & \dot y_C & C&C^c\\
\dot x_D & \dot y_D & D&D^c
\end{bmatrix}\mapsto
\begin{bmatrix}
 A\\  B\\  C\\  D
\end{bmatrix}
$$
with the group action by the group $\rm E(2)$ of Euclidean motions 
(rotations and translations and reflections).
The fibers \eqref{eq:piD} are the orbits of the group.
Each fiber is the full set of all possible locations of a given quadruple in the Euclidean plane.

~
\\
$\bullet$  {\bf Symmetry groups} in $\mathcal D$ and $\mathcal D_{\sf geo}$. 
The space described above is invariant under the action of a group 
that we shall call {\bf Descartes group}, 
$$
    \Desq \ = \ \APq \ \dot\cup \ \SIq
$$
defined by eight generators
that may be split into two quadruplets each defining a subgroup, as described below.
(In this context, they will be refereed to as ``small Descartes groups''.)  
The groups, interestingly, are now linear (hence the lack of the star in the symbol).
\\[-5pt]

\noindent
(1) The {\bf Apollonian group}  $\APq$ is generated by {\bf Descartes moves},
represented by four matrices:
\begin{equation}
\label{eq:4map}
\small
\setlength\arraycolsep{3pt}
M_1 = \begin{bmatrix*}[r]  -1&2&2&2\;\\
                                    0& 1 &0&0\;\\
                                    0  & 0&1&0\; \\
                                   0 & 0 & 0 & 1\;
      \end{bmatrix*}
\quad
M_2 = \begin{bmatrix*}[r] \;1&0&0&0\;\\
                                 2 & \!-1 &2&2\;\\
                                 0 & 1&1&0\;\\
                                 0& 0 &0 &1\;
      \end{bmatrix*}
\quad
M_3 = \begin{bmatrix*}[r] \;1&0& 0&0\;\\
                               0 & 1 &0&0\;\\
                               2 & 2&\!-1&2\;\\
                               0& 0 & 0 &1\;
      \end{bmatrix*}
\quad
M_4 = \begin{bmatrix*}[r]  \;1&0&0& 0\;\\
                               0 & 1 &0& 0\;\\
                               0 & 0&1& 0\;\\
                               2& 2 & 2 &\!-1\;
      \end{bmatrix*}
\end{equation}
Each represents a replacement of one of the disk by 
its conjugation via the remaining three (see \eqref{eq:Descartes2} in Section 3).
It has natural representation of acting on the quadruplets of numbers in $\mathcal D$
or as geometric transformations of the actual disk arrangements in the Euclidean plane.
(Note that it is not equivalent to the tricycle Apollonian group $\APt$ in the previews section.)
The matrices are known as the means of reconstructing  
all disks in the Apollonian packing from a given Descartes configuration, 
and the name ``Apollonian group'' is already standard \cite{Lagarias+}.
\\[-5pt]

\noindent
(2) 
The {\bf kaleidoscope group} $\SIq$ generated by {\bf self-inversions} 
that are now represented by four matrices
\begin{equation}
\label{eq:4mi}
\small
\setlength\arraycolsep{3pt}
N_1 = \begin{bmatrix*}[r]  -1&0&0&0\;\\
                                  2 & 1 &0&0\;\\
                                  2 & 0&1&0\; \\
                                  2 & 0 & 0 & 1\;
      \end{bmatrix*}
      \quad
N_2 = \begin{bmatrix*}[r] \; 1& 2&0&0\;\\
                                 0 &\! -1 &0&0\;\\
                                 0 &  2&1&0\;\\
                                 0& 2 &0 &1\;
      \end{bmatrix*}
\quad
N_3 = \begin{bmatrix*}[r] \; 1&0& 2&0\;\\
                               0 & 1 & 2&0\;\\
                               0 & 0&\!-1&0\;\\
                               0& 0 & 2 &1\;
      \end{bmatrix*}
\quad
N_4 = \begin{bmatrix*}[r] \; 1&0&0& 2\;\\
                               0 & 1 &0& 2\;\\
                               0 & 0&1&2\;\\
                               0& 0 & 0 &\!-1\;
      \end{bmatrix*}
\end{equation}
The matrices represent  self-inversions defined in Section~\ref{sec:si}, now  extended  
from triples to  Des\-cartes configurations.
They can be be viewed as acting on vector $v\in \mathcal D$,
i.e., on quadruples of numbers (algebra),
or on the actual quadruples of disks in the Euclidean plane.
Interestingly, the matrices of the second group
are transposes of the matrices of the first group
despite the different geometric origin and interpretation.

\begin{corollary}
\sf
The Descartes group preserves integrality of symbols and spinors.
\end{corollary}

\noindent
$\bullet$ {\bf Proof of hyper-integrality.} Now we define the transformation that does the trick.

\begin{definition}
\label{def:4descendence}
The {\bf descending} transformation of Descartes quadruples is
the map $T:\mathcal D \to \mathcal D$ defined
\begin{equation}
\label{eq:TT}
T(\mathbf v) = \begin{cases}
                \, M_i\,\mathbf v   & \hbox{if}   \min(\mathbf v) \geq 0 \ \hbox{and} \max(\mathbf v) = v_i\
                \qquad \hbox{\small (descending Descartes move)}\\
                \;N_i\,\mathbf v   & \hbox{if}  \min(\mathbf v) = v_i  < 0   
                \qquad\qquad\qquad\quad  \hbox{\small (descending self-inversion)}
                \end{cases}
\end{equation}
\end{definition}

In words: if the quadruple has a negative term, perform self-inversion with respect to it.
Otherwise, apply the Descartes move with respect to the greatest term.

\begin{proposition}
\label{thm:weight}
\sf
The descending transformation decreases strictly the weight 
$$
  w(T\mathbf \mathbf q) < w(\mathbf q)
$$
exception for $\mathbf q_0=(0,0,1,1)$ and its permutations,
in which case $w(T(\mathbf q_0)) = w(\mathbf q_0)=2$.
\end{proposition}

\begin{proof}
(1) 
Let us use the notation $\mathbf q=(a,b,c,d)$. In case of descending Descartes move 
that applies when $d\geq c\geq b  \geq a \geq 0$, the solution for the fourth disk for the given $a,b,c$ is 
one of the two:  $d_\pm=a\!+\!b\!+\!c\pm\sqrt{ab\!+\!bc\!+\!ca}$. 
Clearly $d_{-} < d_{+}$ with exception of the base quadruple $\mathbf q_0 = (0,0,1,1)$,
and the claim follows as $d_{-}$ replaces $d_{+}$ in the move.
(2) 
In the case of descending self-inversion, the argument is analogous to the tricycle version.  
Say $\mathbf q=(-a,b,c,d)$ where all $a,b,c,d\geq 0$.
$$
w(T\mathbf q) \ = \ w(a, b\!-\!2a, c\!-\!2a, d\!-\!2a)
         \ = \ w(\mathbf q)-4a <w(\mathbf q)
$$
(Note that in the case of all four curvatures positive 
the inequality would fail.)
\end{proof}

\begin{theorem}
\label{thm:Tinverse}
\sf 
All irreducible integral Descartes configurations as quadruples of $\mathcal D$
are hyper-integral, i.e., can be located in the Euclidean plane 
so that they have integral symbols and integral spinors. 
\end{theorem}

\begin{proof}
Starting with a quadruple $v$, the chain of descending transformations
$$
T=T_nT_{n-1}\ldots T_2T_1\qquad T_i\in (the \ set \ of \ eight \ generators)
$$
will bring it to a quadruple $\mathbf q_0=(0,0,1,1)$ (coefficients up to permutation).
Since the weight is strictly decreasing, one must arrive at $\mathbf q_0$.
Now, start with the geometric Descartes configuration 
$$
\small
\bar{\mathbf q}_0 = \begin{bmatrix*}[r]
 -1 & 0 & 0&2\\
  1 & 0 & 0&2\\
  0 & 0 & 1&\!\!\! -1\\
  0 & 2 & 1&3
\end{bmatrix*}
\qquad\qquad
\begin{tikzpicture}[scale=.3,anchor=base,baseline=1ex]  
\draw  [fill=gold!10] (0,0) circle (1);
\draw  [fill=gold!10] (0,2) circle (1);
\draw  [color=white, fill=gold!10] (-2.55,-1.1) rectangle  (-1,3.1);
\draw  [color=white, fill=gold!10] (2.55,-1.1) rectangle  (1,3.1);
\draw [thick] (1,-1.1) -- (1,3.1)
;
\draw [thick] (-1,-1.1) -- (-1,3.1)
;
\end{tikzpicture}
$$
and apply the inverse operation $T^{-1}$ to the configuration $\bar {\mathbf t}_0$ to obtain 
$$
T^{-1}\,\bar{\mathbf q}_0
           \ = \  T_1 T_2 \ldots T_{n-1}T_n  \, \bar{\mathbf q}_0 
$$
where we took into account that for every elementary element we have $T_i^{-1}=T_i$.
Since each elementary $T$ preserves the integrality of the symbols, and
the parity of curvatures and co-curvatures, as well as the tangency spinors,
the claim holds.

\end{proof}

\newpage

\noindent
$\bullet$
{\bf More on the Descartes group.}

~    
    
Basic facts:
\begin{enumerate}
\item
The subgroups $\APq$ and $\SIq$ are, as abstract entities, isomorphic, 
$$
\APq \cong \SIq
$$
or rather anti-isomorphic, since $(ab)^T=b^T\!a^T$.
What distinguishes them is the way we define their action on $\mathcal D$ and on $\mathcal D_{\rm geo}$.

\item
The structure of the spaces: 
The fibers of $(\mathcal D_{\sf geo}, \pi_{\mathcal D}, \mathcal D)$ look like the Bethe lattice
of Figure~\ref{fig:Bethe}. Every vertex in $\mathcal D_{\sf geo}$ has another set of four edges directed 
to vertices of other fibers.
The relation between the groups and fibrations is shown below:
$$
\begin{matrix}
             {\rm E}(2) 
            && \APq \\[-5pt]
            \text{\begin{rotate}{-90}$\rightsquigarrow$\end{rotate}} 
            && \text{\begin{rotate}{-90}$\rightsquigarrow$\end{rotate}}\\[11pt]
             \mathcal D_{\sf geo} 
            &\quad \xrightarrow{~~\pi_{\mathcal D}~~} \quad
            & \mathcal D 
             &\quad  \xrightarrow{~~\pi_{\mathcal A}~~} \quad
             & \mathcal A
             \\[-3pt]
            &&\text{\begin{rotate}{-90}$\cong$\end{rotate}} 
            && \text{\begin{rotate}{-90}$\cong$\end{rotate}}
            \\[7pt]
            &&\mathcal D/E(2)
            && \mathcal D/\APq
\end{matrix}
$$
where the squiggly arrows $\rightsquigarrow$ mean ``a group acts on''.
The situation is analogous to that in Figure~\ref{fig:fiberbundle},
except $\mathcal T$ needs to be replaced by $\mathcal D$.

\item
The inner Descartes group acts transitively in the set of all 
Descartes quadruples (configurations) of a given Apollonian packing.
In other words, they preserve the fibers of fibration
$$
\pi_{\mathcal D}: \mathcal D\quad \to \quad \mathcal  D/\APq \cong \mathcal A  
$$
The factor space coincides with the set $\mathcal A$ of irreducible Apollonian disk packings.
Elements of the group $\SIq$ mix the fibers, but not consistently, i.e. in general
$$
       \pi_{\mathcal A}(\mathbf q) \not =         \pi_{\mathcal A}(g\mathbf q) 
             \qquad g\in \SIq,\quad  g\not= \id
$$         

\item
The Descartes group is orthogonal: 
$$
\Desq \subset {\rm O}(\mathbb R^4, G) \cong O(1,3; \mathbb R)
$$
as each element $X$ of the generating set satisfies  
\begin{equation}
\label{eq:ortho}
X^TGX=G     \,,
\end{equation}
and so do their products as well.

\item
Each generator is an  involution, i.e., satisfies $X^2 = \id$.
Geometrically, it represents a reflection in the Minkowski space.  
The reflection is orthogonal with respect to $G$.
The associated axis $\mathbf A$ 
and the invariant hyperplane $\mathbf A^\bot$ 
is shown here for the first generators:

For $M_1$: \quad  $\mathbf A_1 = [ 1, 0, 0, 0]$, \quad $\mathbf A_1^\bot =\span( [x\!+\!y\!+\!z,x,y,z]:x,y,z\in\mathbb R)$

For $N_1$:  \quad $\mathbf B_1 = [-1,1,1,1]$, \quad 
$\mathbf B_1^\bot = \span([ 0, x, y, z ] :x,y,z\in\mathbb R)$

\item
The generators represent orthogonal reflection with respect to the inner product $G$.
They be defined from the corresponding eigenvectors $\mathbf v$ 
with the eigenvalue $\lambda=-1$   (call the ``axes'') as follows:
$$
M_{\mathbf v} = \id -2\, \frac {\mathbf v \, \mathbf v^TG}{\mathbf v^TG\,\mathbf v}
$$
For the matrices $M$, the axes are $[1,0,0,0]$ and the permutations.
For the matrices $N$ the axes are $[-1,1,1,1]$ and the permutations.

\item
Note that the only eigenvectors of the generators
that are contained in the set $\mathcal D$ are 
$$
{\small
[1,\!1,\!0,\!0], \   
[1,\!0,\!1,\!0], \   
[1,\!0,\!0,\!1], \   
[0,\!1,\!1,\!0], \ 
[0,\!1,\!0,\!1], \   
[0,\!0,\!1,\!1].
}
$$ 
with the eigenvalue $\lambda = 1$, which are different versions of the base state $\mathbf q_0$.

\item
The topology of the groups:  The Cayley graph of each of the ``small Descartes'' groups 
$\SIq$ and $\APq$ is 
similar to that of the free group of rank 2  
(but not equivalent, for we have relations $X^2=\id$ for the generators).  
The graph is also known under the name of ``$z=4$ Bethe lattice'', shown in Figure~\ref{fig:Bethe}.
\begin{figure}[h]
\centering
\begin{tikzpicture}[scale=.3] 
\draw  [fill=black] (0,0) circle (.5);
\draw  [fill=black] (4,0) circle (.4);
\draw  [fill=black] (0,4) circle (.4);
\draw  [fill=black] (-4,0) circle (.4);
\draw  [fill=black] (0,-4) circle (.4);
\draw  [fill=black] (4,1.5) circle (.3);
\draw  [fill=black] (4,-1.5) circle (.3);
\draw  [fill=black] (6,0) circle (.3);
\draw  [fill=black] (-4,1.5) circle (.3);
\draw  [fill=black] (-4,-1.5) circle (.3);
\draw  [fill=black] (-6,0) circle (.3);
\draw  [fill=black] (1.5,4) circle (.3);
\draw  [fill=black] (-1.5,4) circle (.3);
\draw  [fill=black] (0,6) circle (.3);
\draw  [fill=black] (1.5,-4) circle (.3);
\draw  [fill=black] (-1.5,-4) circle (.3);
\draw  [fill=black] (0,-6) circle (.3);

\draw [thick] (-7.4 , 0) -- (7.4,0);
\draw [thick] (0,-7.4) -- (0,7.4);

\draw [thick] (-3 , 4) -- (3,4);
\draw [thick] (-3 , -4) -- (3,-4);
\draw [thick] (4, -3) -- (4, 3);
\draw [thick] (-4,-3) -- (-4, 3);

\draw [] (6,-1) -- (6, 1);
\draw [] (-6,-1) -- (-6, 1);
\draw [] (-1,6) -- (1,6);
\draw [] (-1,-6) -- (1,-6);

\draw [] (3,1.5) -- (5, 1.5);
\draw [] (3,-1.5) -- (5, -1.5);
\draw [] (-3,1.5) -- (-5, 1.5);
\draw [] (-3,-1.5) -- (-5, -1.5);

\draw [] (1.5,3) -- (1.5, 5);
\draw [] (-1.5,3) -- (-1.5, 5);
\draw [] (1.5,-3) -- (1.5, -5);
\draw [] (-1.5,-3) -- (-1.5, -5);
\end{tikzpicture}
\caption{Cayley graph of each small Descartes group is an infinite 4-valent cycle-free graph, 
also known as a $z=4$ Bethe lattice.}
\label{fig:Bethe}
\end{figure}
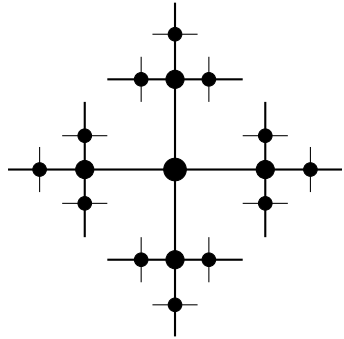

The ``big'' Descartes group $\Desq$ is a graph of valency 8, but it is not  cycle-free anymore, 
since the following identities hold:
$$
M_iN_j = N_jM_i  \qquad \hbox{whenever} \  i\not = j
$$
or, $[M_i,\, N_j] = 0$.  Hence, e.g., $M_1N_2M_1N_2=\id$ defines a cycle in the Cayley graph.

\item
The intriguing fact that the matrix generators of the two small 
Descartes groups are transposes of each other can be interpreted 
with rewriting the orthogonality property \eqref{eq:ortho}:
$$
N_iGM_i = G       \qquad \hbox{and} \qquad M_iGN_i = G   \qquad \forall i=1,2,3,4.  
$$

\item
{\bf Remark:}  The fact that the matrices 
representing the two small groups are transposes of each other is accidental
and inherent to the 2-dimensional instance. 
In 3-dimensional case of 3-disks (balls) in $\mathbb R^3$,
the  {\bf kaleidoscope group} $\SIqq$ is generated by five matrices analogous to the ones before
\begin{equation}
\label{eq:4M5x5}
\small
\setlength\arraycolsep{3pt}
 \begin{bmatrix*}[r]  -1&0&0&0&0\;\\
                                  2 & 1 &0&0&0\;\\
                                  2 & 0&1&0&0\; \\
                                  2 & 0 & 0 & 1&0\;\\
                                  2 & 0 & 0 &0& 1\;
      \end{bmatrix*}\,,
      \quad
\begin{bmatrix*}[r] \; 1& 2&0&0&0\;\\
                                 0 &\! -1 &0&0&0\;\\
                                 0 &  2&1&0&0\;\\
                                 0& 2 &0 &1&0\;\\
                                 0& 2 &0 &0&1\;

      \end{bmatrix*}\,,
\quad
\begin{bmatrix*}[r] \; 1&0& 2&0&0\;\\
                               0 & 1 & 2&0&0\;\\
                               0 & 0&\!-1&0&0\;\\
                               0& 0 & 2 &1&0\;\\
                               0& 0 & 2 &0&1\;

      \end{bmatrix*}\,,
\quad
\ldots \ etc.
\end{equation}
but the {\bf Apollonian} $\APqq$  is generated by five $5\times 5$  matrices:
\begin{equation}
\setlength\arraycolsep{3pt}
\small
\begin{bmatrix*}[r]  -1&1&1&1&1\;\\
                                    0& 1 &0&0&0\;\\
                                    0  & 0&1&0&0\; \\
                                   0 & 0 & 0 & 1&0\;\\
                                   0 & 0 & 0 & 0&1\;
      \end{bmatrix*}\,,
\quad
\begin{bmatrix*}[r] \;1&0&0&0&0\;\\
                                 1 & \!-1 &1&1&1\;\\
                                 0 & 0&1&0&0\;\\
                                 0& 0 &0 &1&0\;\\
                                   0 & 0 & 0 & 0&1\;
      \end{bmatrix*}\,,
\quad
\begin{bmatrix*}[r] \;1&0& 0&0&0\;\\
                               0 & 1 &0&0&0\;\\
                               1 & 1&\!-1&1&0\;\\
                               0& 0 & 0 &1&0\;\\
                               0& 0 & 0 &0&1\;\\
       \end{bmatrix*}\,,
\quad
\ldots \ etc.
%
\end{equation}
which clearly not mutual transposes.

\end{enumerate}

~

\noindent
$\bullet$ {\bf Higher dimensions} 
\\\\
For the sake of completeness, here is the characterization of the Descartes groups 
for $n$-disks in $n$-dimensional space, represented by $m\times m$ matrices, $m=n+2$.
We use the standard notation where $\mathbf e_i$ denotes the column with 1 at the $i$-th position 
and zeros everywhere else.  
Similarly, $\mathbf e_{ij}$ is a matrix with 1 at the $(ij)$-th position and zeros everywhere else.

~


Consider the pseudo-Euclidean space $(\mathbb R^m, G)$, where 
the inner product is given by matrix 
$$
G \ = \ \sum_{i=1}^n (3-m)\,\mathbf e_{ii} + \sum_{i\not= j} \mathbf e_{ij}  \,.
$$
There are three groups generated by $m\times m$ matrices, denoted:
$$
\begin{array}{crl}
(1)  &  {\AP}(m)\!\! &=\; \gen \{\,M_i \;\big|\; i=1,\ldots,m\,\} \, \quad  \\
(2)  &  {\SI}(m) \!\! &=\; \gen \{\,N_i\;\big|\; i=1,\ldots,m\;\}  \,,\quad  \\
(3) &  {\Des}(m) \!\! &=\; \gen \{\,M_i, \, N_i\;\big|\; i=1,\ldots,m\;\}  
\end{array}
$$
where the matrices are  
$$
M_i=\id - 2\mathbf e_{ii} + \frac{2}{m-3} \sum_{j,\, j\not =i}\mathbf e_{ij} \,,
\qquad 
N_i= \id - 2\mathbf e_{ii} + 2 \sum_{j,\, j\not =i}\mathbf e_{ji} \,.
$$
Here are the first matrices of each subgroup: 
\begin{equation}
\label{eq:matrHdim}
\small
\setlength\arraycolsep{3pt}
M_1 = \begin{bmatrix*}[r]  -1&\frac{2}{m-3}&\ldots&\frac{2}{n-1}\;\\
                                    0& 1\ \ &\ldots&0\ \ \;\\
                                    \vdots  & &&\vdots\ \ \; \\
                                   0 & 0 \ \ & \ldots & 1\ \ \;
      \end{bmatrix*}
\qquad 
N_1 = \begin{bmatrix*}[r]  -1&0&\ldots&0\;\\
                                  2 & 1 &\ldots&0\;\\
                                  \vdots & &&\vdots\; \\
                                  2 & 0 & \ldots & 1\;
      \end{bmatrix*}
\end{equation}
Basic properties:  The matrices are ({\it i}) involutions
and ({\it ii}) orthogonal with respect to~$G$:
$$
\begin{array}{cl}
(i) & M_i^2 = N_i^2 = \id  \qquad \forall i=1,\ldots,n \\
(ii) & X^tGX=G\\
(iii) & \det X = -1
\end{array} 
$$
In general, the generators do not commute, except the case:
$$
\hbox{If} \quad i\not= j \quad\hbox{then} \quad  [M_i, N_j] = 0  
$$
The generators are orthogonal reflections in hyperplanes and may be 
characterized by   
axes (eigenvectors with eigenvalues $\lambda=-1$):
$$
\mathbf A_i = \mathbf e_i \,,
\qquad 
\mathbf B_i = - \mathbf e_i + \sum_{j\not=i} \mathbf e_j \,.
$$
The generators can be thus defined as:
$$
M_{i} \ = \   \id -2\, \frac {\tensor{\mathbf A}{_i}\tensor{\mathbf A}{_i^T} G}%
                                      {\,\tensor{\mathbf A}{_i^T} G\,\tensor{\mathbf A}{_i}}\,,
\qquad 
N_{i} \ = \ \id -2\, \frac {\tensor{\mathbf B}{_i}\tensor{\mathbf B}{_i^T}G}
                                   {\,\tensor{\mathbf B}{_i^T} G\,\tensor{\mathbf B}{_i}}   \,.
$$

\noindent
{\bf Small dimensions:}  The cases of $m=4$ and $m=5$ are given above. 
The case of $m=3$ corresponding to 1-disks on a real line is degenerated 
and only the component $\SI(3)$ is well-defined.
The integrality of the matrices \eqref{eq:matrHdim} breaks for dimensions $m>5$.

Geometrically, the case $m=2$ corresponds to the absurd case of the 0-disks 
in the 0-dimensional space,
yet formally it is well-defined; the matrices are:
$$
\small
\setlength\arraycolsep{3pt}
G = \begin{bmatrix*}[r] \, 1&1\,\\
                                     1& 1\, \\
      \end{bmatrix*},
      \quad 
M_1 = \begin{bmatrix*}[r]  -1&-2\,\\
                                          0& 1\, \\
      \end{bmatrix*}\,,
\  
M_2 = \begin{bmatrix*}[r]  1&0\,\\
                                       -2& -1\, \\
      \end{bmatrix*}\,,
\  
N_1 = \begin{bmatrix*}[r]  -1&0\,\\
                                         2 & 1\, \;
      \end{bmatrix*}\,,
\  
N_2 = \begin{bmatrix*}[r] \, 1&2\,\\
                                        0 & -1\,
      \end{bmatrix*}.
$$
where $G$ is the Gram matrix of the inner product in the space. 
Note that the first matrix determines the others:  $M_2= - M_1^T$,
and the $N$ matrices are the negatives of these two.
They produce an interesting subgroup of $\hbox{SL}^{\pm}(2)$.

~\\
{\bf Remark:}  For additional material see pages at \cite{jk-www}.

\section*{Acknowledgments}

I would like to thank Philip Feinsilver for helpful discussions 
and his supportive interest in this project.


%
%
%
%
%
%
%
%
%
%
%

\end{document}